\documentclass{amsart}
\usepackage[all]{xy}
\usepackage{comment}
\usepackage{graphicx}
\usepackage{caption}
\usepackage{subcaption}
\usepackage{xcolor}
\usepackage{url}
\usepackage[hidelinks,colorlinks=true,linkcolor=blue, citecolor=black,linktocpage=true]{hyperref}
\usepackage{tikz}
\usepackage{pgfplots}
\usepackage{thmtools}
\usepackage{thm-restate}
\usepackage{cleveref}
\usepackage{amssymb}

\def\R{\mathbb{R}}

\def\N{\mathbb{N}}

\def\cD{\mathcal{D}}

\def\mylist#1 {\ifx!#1\else\makebox[4em][l]{#1} \expandafter\mylist\fi}

\declaretheorem[name=Theorem,numberwithin=section]{thm}

\theoremstyle{definition}
\newtheorem{defn}[thm]{Definition}
\theoremstyle{definition}
\newtheorem{rem}[thm]{Remark}

\newtheorem{conj}[thm]{Conjecture}

\begin{document}
\title[The unknotting number and Reinforcement Learning]{The unknotting number, hard unknot diagrams, and reinforcement learning}

\author[Applebaum]{Taylor Applebaum}
\address{Google DeepMind, London, UK}
\email{applebaum@google.com}

\author[Blackwell]{Sam Blackwell}
\address{Google DeepMind, London, UK}
\email{blackwells@google.com}

\author[Davies]{Alex Davies}
\address{Google DeepMind, London, UK}
\email{adavies@google.com}

\author[Edlich]{Thomas Edlich}
\address{Google DeepMind, London, UK}
\email{edlich@google.com}

\author[Juh\'asz]{Andr\'{a}s Juh\'{a}sz}
\address{Mathematical Institute, University of Oxford, Andrew Wiles Building,
	Radcliffe Observatory Quarter, Woodstock Road, Oxford, OX2 6GG, UK}
\email{juhasza@maths.ox.ac.uk}

\author[Lackenby]{Marc Lackenby}
\address{Mathematical Institute, University of Oxford, Andrew Wiles Building,
	Radcliffe Observatory Quarter, Woodstock Road, Oxford, OX2 6GG, UK}
\email{lackenby@maths.ox.ac.uk}

\author[Toma\v{s}ev]{Nenad Toma\v{s}ev}
\address{Google DeepMind, London, UK}
\email{nenadt@google.com}

\author[Zheng]{Daniel Zheng}
\address{Google DeepMind, London, UK}
\email{dhhzheng@google.com}


\begin{abstract}
	We have developed a reinforcement learning agent that often finds a minimal sequence of unknotting crossing changes for a knot diagram with up to 200 crossings, hence giving an upper bound on the unknotting number. We have used this to determine the unknotting number of 57k knots. We took diagrams of connected sums of such knots with oppositely signed signatures, where the summands were overlaid. The agent has found examples where several of the crossing changes in an unknotting collection of crossings result in hyperbolic knots. Based on this, we have shown that, given knots $K$ and $K'$ that satisfy some mild assumptions, there is a diagram of their connected sum and $u(K) + u(K')$ unknotting crossings such that changing any one of them results in a prime knot. As a by-product, we have obtained a dataset of 2.6 million distinct hard unknot diagrams; most of them under 35 crossings. Assuming the additivity of the unknotting number, we have determined the unknotting number of 43 at most 12-crossing knots for which the unknotting number is unknown. 
\end{abstract}

\maketitle

\section{Introduction}

Knot theory plays a fundamental role in low-dimensional topology. A \emph{knot} is a smooth embedding $K \colon S^1 \hookrightarrow S^3$. We say that the knots $K$ and $K'$ are \emph{equivalent} if there is an orientation-preserving automorphism $\varphi$ of $S^3$ such that $\varphi \circ K = K'$. We can represent a knot using a projection onto $S^2$ with only transverse double point singularities, together with information at each double point about which strand is higher with respect to the projection. This is called a \emph{knot diagram}. Two knot diagrams represent the same knot if and only if they are related by a sequence of Reidemeister moves R1--R3. For textbooks on knot theory, see Burde--Zieschang~\cite{Burde-Zieschang}, Lickorish~\cite{Lickorish-book}, and Rolfsen~\cite{Rolfsen-book}.

\subsection{The unknotting number.} The unknotting number is one of the oldest and most natural, yet most elusive knot invariants. The \emph{unknotting number $u(\cD)$ of a knot diagram} $\cD$ is the minimal number of crossing changes required to obtain a diagram of the unknot $U$. The \emph{unknotting number $u(K)$ of a knot $K$} is defined as 
\[
u(K) := \min\{u(\cD) : \cD \text{ is a diagram of } K \}.
\]
Taniyama~\cite{Taniyama} has shown that, given any knot $K$ and $n \in \N$, there is a diagram $\cD$ of $K$ with $u(\cD) \ge n$.

A more intrinsic definition of the unknotting number is obtained using crossing arcs. A \emph{crossing arc} $a$ for a knot $K$ is a framed, oriented arc smoothly embedded in $S^3$ such that $K \cap a = \partial a$. Here, a framing of $a$ is a trivialisation $(v_1, v_2)$ of the normal bundle of $a$ such that $v_1(a(0))$ and $v_2(a(1))$ are positive tangents to $K$. A crossing change along $a$ is obtained by performing a finger move on $K$ along $a$. More precisely, for some $\varepsilon > 0$ small, we replace a small segment of $K$ passing through $a(0)$ with curves $\varepsilon v_1(a(t))$ and $- \varepsilon v_1(a(t))$ for $t \in I$, together with a semi-circular arc connecting $\varepsilon v_1(a(1))$ and $-\varepsilon v_1(a(1))$. Then $u(K)$ is the minimal number of such crossing changes that result in $U$. If $a_1, \dots, a_n$ is a collection of crossing arcs, then we can always isotope them and the knot $K$ so that the crossing arcs become short, vertical line segments with parallel framing, in which case the crossing changes along them correspond to crossing changes in the respective diagram. This shows the equivalence between the two definitions. Yet another description of the unknotting number is given by taking the minimal number of double points that appear in a generic regular homotopy from $K$ to $U$; see Lickorish~\cite[p.~7]{Lickorish-book}.

Another classical knot invariant that can be defined in an analogous manner is the crossing number. The \emph{crossing number $c(\cD)$ of a knot diagram} $\cD$ is the number of double points of $\cD$. The \emph{crossing number $c(K)$ of a knot $K$} in $S^3$ is defined as
\[
c(K) := \min\{c(\cD) : \cD \text{ is a diagram of } K \}.
\]
If $c(\cD) = c(K)$, we say that $\cD$ is a \emph{minimal crossing number diagram of $K$}. Note that a knot can have several inequivalent minimal crossing number diagrams. There are finitely many knots for each crossing number, so knots are usually tabulated by crossing number. See Rolfsen~\cite{Rolfsen-book} for a table of knots up to ten crossings and KnotInfo~\cite{knotinfo} for knots up to 13 crossings. For example, the only crossing number zero knot is the \emph{unknot} $U$ and there are no knots of crossing number one or two. There are two knots of crossing number three, the right-handed trefoil $3_1$ (often denoted by $T_{2,3}$, as it is the $(2,3)$-torus knot) and its mirror, the left-handed trefoil $-3_1$. There is only one knot of crossing number four, the figure eight knot $4_1$, which is equivalent to its mirror. The number of knots grows exponentially as the crossing number increases; see Welsh~\cite{Welsh}. For an integer $c \in [3,10]$, the notation $c_n$ refers to knot number $n$ of crossing number $c$ in Rolfsen's table.

There is no algorithm known to compute $u(K)$. The main difficulty is that there are knots $K$ such that $u(K) < u(\cD)$ for any minimal crossing number diagram $\cD$ of $K$. 
For example, the knot $10_8$ has a unique minimal crossing number diagram $\cD$ with $u(\cD) = 3$, but $u(10_8) = 2$. If one changes a suitable crossing of $\cD$, one obtains a 10-crossing diagram of $6_2$, and $u(6_2) = 1$. By applying random Reidemeister moves to $\cD$, it is easy to find a diagram $\cD'$ of $10_8$ with $u(\cD') = 2$.
 
Of the 2978 knots with at most 12 crossings, 660 have unknown unknotting number, including 9 knots with crossing number 10. Of the remaining 2318 knots, only 25 have $u(\cD) > u(K)$ for their minimal crossing number diagram $\cD$ in KnotInfo~\cite{knotinfo}.

A conjecture of Bernhard~\cite{Bernhard} and Jablan~\cite{Jablan} stated that every knot $K$ has a minimal crossing number diagram $\cD$ and a crossing $c$ such that changing $c$ results in a knot $K'$ with $u(K') = u(K) - 1$. If true, this would yield an algorithm for computing $u(K)$: Enumerate all minimal crossing number diagrams of $K$, list all diagrams that can be obtained from these by changing one crossing, and repeat with these knots until we first reach the unknot. The number of steps is $u(K)$.  However, Brittenham and Hermiller~\cite{Brittenham-Hermiller} have shown this to be false: at least one of 12n288, 12n491, 12n501, and 13n3370 violates the conjecture. One can obtain each of 12n288, 12n491, and 12n501 from 13n3370 via a single crossing change. The knot 13n3370 is the closure of a 20-crossing braid, where changing a single crossing gives 11n21 that has unknotting number one. So $u(13n3370) \le 2$, but it is hard to find a diagram $\cD$ of 13n3370 with $u(\cD) = 2$ by applying random Reidemeister moves to a minimal crossing number diagram. 

The \emph{Gordian graph} $G$ has vertices knots, and an edge connects two knots if they are related by a crossing change. The \emph{Gordian distance} $d(K, K')$ of the knots $K$ and $K'$ is the distance of $K$ and $K'$ in $G$. Using this notion, $u(K) = d(K, U)$. Baader~\cite{Baader} has shown that, if $d(K, K') = 2$, then there are infinitely many knots $K''$ such that 
\[
d(K,K'') = d(K', K'') = 1.
\]
Hence, the number of minimal unknotting trajectories for a knot is typically infinite.

In practice, a good upper bound on $u(K)$ can often be obtained by first simplifying a diagram of $K$ (removing crossings using Reidemeister moves) and then selecting a crossing to change by picking one that results in a diagram that can be further simplified to have the fewest crossings. This process is then repeated iteratively.

The \emph{4-ball genus} $g_4(K)$ of $K$ is the minimal genus of a compact, connected, and oriented surface smoothly embedded in the 4-dimensional unit ball $D^4$  with boundary $K$. It satisfies $g_4(K) \le u(K)$. Indeed, the trace of a generic regular homotopy of $K$ to the unknot with $u(K)$ transverse double points gives rise to an immersed disc in $D^4$ with $u(K)$ transverse double points and boundary $K$. If we smooth these double points, we obtain a compact, connected, and oriented surface of genus $u(K)$ in $D^4$ with boundary $K$. 

Most known lower bounds on $u(K)$ are also lower bounds on $g_4(K)$.
Among these, $|\sigma(K)|/2$, $|\tau(K)|$, $|\nu(\pm K)|$ (where $-K$ is the mirror of $K$), and $|s(K)|/2$ are efficiently computable, 
where $\sigma$ is the signature~\cite[Proposition~4.28]{Juhasz-book}, the invariants $\tau$ and $\nu$  were defined and shown to bound $g_4(K)$ by Ozsv\'ath and Szab\'o~\cite{OSz-4-ball} using knot Floer homology~\cite{OSz-HFK}\cite{Ras}, and $s$ was defined and shown to bound $g_4(K)$ by Rasmussen~\cite{Rasmussen-s} via Khovanov homology~\cite{Khovanov-homology}. (Note that $\sigma$, $\tau$, and $s$ are symmetric with respect to mirroring, but $\nu$ is not.)
We obtain $u(K)$ if the upper and lower bounds agree. For example,
Kronheimer and Mrowka~\cite{Kronheimer-Mrowka-Milnor} have shown that, for the torus knot $T_{p,q}$, we have
\[
u(T_{p,q}) = \frac{(p-1)(q-1)}{2}.
\]
In general, there are few classes of knots for which the unknotting number is known.
See~\cite{Lackenby-survey} for a survey of results on the unknotting number.

\subsection{Additivity of unknotting number}
An old open question is whether the unknotting number is additive under connected sum.

\begin{conj}\label{conj:additivity}
For knots $K$ and $K'$, we have $u(K \# K') = u(K) + u(K')$. 
\end{conj}

There is very little theoretical evidence to support this conjecture.
Scharlemann~\cite{Scharlemann-unknotting} has shown that $u(K \# K') \ge 2$ if $K$, $K' \neq U$.
More recently, Alishahi and Eftekhary~\cite{Alishahi-Eftekhary-unknotting} have proven using knot Floer homology that 
\[
u(K \# T_{p,q}) \ge p-1
\]
for integers $0 < p < q$.
However, these results leave open the possible existence of knots $K$, $K'$ for which $u(K)$ and $u(K')$ are both large but where $u(K \# K') = 2$.

We therefore endeavoured to find counterexamples to Conjecture \ref{conj:additivity}. Although we were not successful, we discovered a large amount of new and interesting information about unknotting number and about knot diagrams.

To find counterexamples to the conjecture, one needs to start with knots $K$ and $K'$ with known unknotting numbers, and then to find efficient ways of unknotting $K \# K'$. One significant source of knots $K$ with known unknotting number is those for which $u(K) = |\sigma(K)|/2$. Given two such knots $K$ and $K'$, then of course $u(K \# K') = u(K) + u(K')$ if $\sigma(K)$ and $\sigma(K')$ have the same signs, as $\sigma$ is additive under connected sum. However, if they have opposite signs, then there is no obvious reason why $K \# K'$ cannot be a counterexample to the conjecture. A further source of knots with known unknotting number are torus knots, and again there seems to be no known reason why $u(K \# K') = u(K) + u(K')$ for torus knots of signature with opposite signs.
Indeed, it is currently unknown whether  
\[
u(T_{2,3} \# -T_{2,5}) = u(T_{2,3}) + u(T_{2,5}).
\]

One other reason to doubt Conjecture \ref{conj:additivity} is the apparent absence of any plausible potential method for proving it. One possible approach might be to establish the following stronger conjecture.

\begin{conj}\label{conj:strong}
	In any collection of unknotting crossing arcs for $K \# K'$, there is one arc that can be isotoped to be disjoint from the 2-sphere specifying the connected sum.
\end{conj}

This implies Conjecture \ref{conj:additivity} by a simple induction on $u(K \# K')$. (See Section~\ref{sec:strong} for this implication.) However, we were able to find counterexamples to this conjecture, which we will describe below. Note that it is also open whether the crossing number is additive under connected sum, though this is widely believed to be true.

\subsection{Finding efficient unknotting sequences}
A crucial part of the strategy for disproving Conjecture \ref{conj:additivity} is to be able to find short unknotting sequences. In particular, in the case of $K \# K'$, the number of crossing changes needs to be less than $u(K) + u(K')$. Even when one is presented with a diagram $\cD$ for a knot $K$, it is not straightforward to compute $u(\cD)$ when the crossing number of $\cD$ is large.
For a knot diagram $\cD$ with $n$ crossings, $u(\cD) \le n/2$. If a set of crossings is unknotting, so is its complement. Hence, there are $2^{n-1}$ possibilities for the subset of crossings that yield a diagram of the unknot, up to taking complements.
This makes computing $u(\cD)$ practically impossible when $n$ is large. 

In order to find out which knot invariants to use for our reinforcement learning experiments, we first trained a supervised learning model on brute-forced unknotting sets that predicts the probability a given crossing lies in a minimal unknotting set. This is an instance of behavioural cloning, the simplest form of imitation learning. This performed well above baseline, and the most useful feature was the Jones polynomial.

We then trained a reinforcement learning agent that can efficiently find an unknotting sequence of crossing changes in a diagram with as many as 200 crossings. Given the small amount of initial training data, this was initially evaluated on a brute-forced dataset of diagrams. Thereafter, we used unknotting sets provided by the agent to evaluate progress.

We have used various features to aid the reinforcement learning agent, and again found the Jones polynomial to be by far the most useful. This suggests that the Jones polynomial contains yet unobserved information about the unknotting number.

By combining the agent with lower bounds coming from invariants such as the signature, $\tau$, $\nu$, and $s$, we have obtained a dataset of about 57k knot diagrams with known unknotting numbers. We have then taken connected sums of such diagrams, which were overlaid and, in some cases, then randomly mixed using Reidemeister moves. We have also run it on connected sums of braid closures that were mixed by inserting subwords representing the trivial braid. The agent found unknotting sequences that involved several crossing changes that resulted in hyperbolic knots, and were hence not connected sums. Throughout this paper, we will use the terminology that a crossing change in a diagram of a connected sum is \emph{inter-component} if it results in a knot that is not a connected sum (e.g., hyperbolic) and is \emph{in-component} otherwise. This led us to diagrams of connected sums of knots $K$ and $K'$ that admit an unknotting subset of crossings of size $u(K) + u(K')$, such that any single crossing change from the unknotting subset results in a hyperbolic knot, hence disproving Conjecture~\ref{conj:strong}. In fact, we will prove the following:

\begin{restatable}{thm}{strong} 
\label{thm:strong}
	Suppose that the prime knots $K_1$ and $K_2$ in $S^3$ are not 2-bridge. Suppose that, for $i \in \{1,2\}$, there is a set of $u(K_i)$ crossing changes to $K_i$ taking it to the unknot, with the property that changing any one of these crossings does not produce the connected sum of $K_i$ and a non-trivial knot. Furthermore, assume that $u(K_1) > 1$ or $u(K_2) > 1$. Then there is a diagram of $K_1 \# K_2$ and a set $C$ of unknotting crossings of size $u(K_1) + u(K_2)$ such that changing any crossing in $C$ results in a prime knot. 
\end{restatable}

We suspect that, in the above theorem, it is not necessary to assume that the summands are not 2-bridge. Note that it is very reasonable to make the hypothesis about the existence of $u(K_i)$ crossing changes as in the statement of the theorem. Certainly, $K_i$ has a sequence of $u(K_i)$ crossing changes taking it to the unknot, and if we change any of these crossings, the result is a knot $K'$ with $u(K') = u(K_i) -1$. So if $K'$ had $K_i$ as a summand, then this would contradict Conjecture \ref{conj:additivity} and hence Conjecture \ref{conj:strong}. 

However, even after running the agent on millions of connected sums, we have not found a counterexample to the additivity of the unknotting number. 

\subsection{New unknotting numbers, assuming additivity.}
Conjecture~\ref{conj:additivity} has interesting consequences for the unknotting number of some prime knots. Suppose that we have a sequence of unknotting crossing changes of length $u(J)$ for a knot $J$.
Then, if we change $n$ of these crossings, the resulting knot
must have unknotting number $u(J) - n$. Hence, if we start with a knot $K \# K'$ and find a sequence of $u(K) + u(K')$ crossing changes that takes it to the unknot, then, assuming Conjecture \ref{conj:additivity}, we can determine the unknotting number of all the intermediate knots in the sequence. Using this approach, we have obtained 43 at most 12-crossing prime knots with unknown unknotting numbers. This provides a method for computing the unknotting numbers of these 43 knots, assuming Conjecture~\ref{conj:additivity}. These 43 values all coincide with the largest possible unknotting number given in the KnotInfo database.
Conversely, if one of these at most 12-crossing prime knots had smaller unknotting number than the KnotInfo upper bound, we would obtain a counterexample to the additivity of the unknotting number.

\subsection{Hard unknot diagrams.} It is a major open problem in knot theory whether there is a polynomial-time unknot detection algorithm. We say that a diagram of the unknot is \emph{hard} if, in any sequence of Reidemeister moves to the trivial diagram, the crossing number has to first increase before it decreases. They are of particular interest because they might provide counterexamples to potential unknot detection algorithms. Hard unknot diagrams are difficult to construct, and previously no extensive dataset existed. Burton, Chang, L\"offler, Mesmay, Maria, Schleimer, Sedgwick, and Spreer~\cite{hard} have recently collected 21 hard unknot diagrams and 2 special infinite families from the literature, 10 of which are not actually hard according to our definition, as they can be simplified without increasing the crossing number (though a monotonically decreasing simplification might not exist). 

Initially, we tried to construct hard unknot diagrams using reinforcement learning, where a setter performs complicating Reidemeister moves to prevent a solver from unknotting via simplifying Reidemeister moves, with little success.

During our unknotting  experiments, we have found approximately 5.9 million knot diagrams between 9 and 75 crossings that SnapPy could not simplify even after 25 attempts. We have shown that 2.46 million of these are indeed hard and are not related by a sequence of R3 moves. Some of these have thousands of R3-equivalent diagrams; see Figure~\ref{fig:large-moduli}. Our dataset includes the first four hard unknot diagrams H, J, Culprit, and Goeritz from \cite{hard} of crossing numbers 9, 9, 10, and 11, respectively. The next previously known hard unknot diagram, the reduced Ochiai~II, has 35 crossings. The vast majority of the hard unknot diagrams that we have found have less than 35 crossings. 

\section{Some background on Machine Learning}

There are three major Machine Learning paradigms, namely, \emph{supervised learning} (SL), \emph{reinforcement learning} (RL), and \emph{unsupervised learning}. In this paper, we will focus on the first two.

In SL, we are given a labelled dataset. In other words, we know the values of a function at certain points. We split our dataset into a \emph{training set}, which is typically about 80\%, and a \emph{test set}. We would like to learn, or approximate, the function using only the training set such that the error (e.g., $L^2$-norm) is small on the whole dataset. 

The most classical example is linear regression. More generally, Hornik, Stinchcombe, and White~\cite{universal-approximation} have shown that  \emph{neural networks} (NNs) are universal function approximators, if one is allowed to vary the architecture. A neural network is a composition of a sequence of affine maps and some simple non-linearities in between, such as $\max(0, x)$ applied coordinate-wise. The network is trained using some variant of \emph{stochastic gradient descent}. One initialises the affine maps, for example, randomly, then computes an approximation of the gradient of the error on a subset of the training set (whose cardinality is called the \emph{batch size}), and changes the affine maps in the direction of the negative gradient according to some \emph{step size} (or \emph{learning rate}). This is repeated a number of times, and a pass through the whole training set is called an \emph{epoch}.

There have been several applications of SL to knot theory in recent years, mostly aimed at finding connections between knot invariants. See, for example, Hughes~\cite{Hughes} and Davies et al.\,\cite{Nature}.

RL is a machine learning paradigm where an agent (in our case, computer software) learns to perform actions to maximise a cumulative reward while interacting with an environment. Typical examples are provided by the games of chess and Go, self-driving cars, and humanoid robots that learn to walk. Training a SL model is often much simpler than RL. There have been only two applications of RL to topology so far. Gukov, Halverson, Ruehle, and Su{\l}kowski \cite{unknotting} focused on unknot recognition. Furthermore, Gukov, Halverson, Manolescu, and Ruehle~\cite{ribboning} have developed RL agents that search for ribbon disks for a knot. In this rest of this section, we give an overview of RL and imitation learning.

\subsection{Markov decision processes.} 
Mathematically, RL is a framework which can be used to find solutions to problems that are formulated as \emph{Markov decision processes}. A Markov decision process is a tuple $(S, A_s, P_a, R_a)$, where
\begin{itemize}
	\item $S$ is a set of \emph{states}, 	
	\item $A_s$ is the set of \emph{actions} available from state $s \in S$,
	\item $P_a(s, s')$ is the \emph{probability} that $a \in A_s$ leads to state $s' \in S$, and
	\item $R_a(s, s')$ is the immediate \emph{reward} after transitioning from state $s$ to $s'$ via action $a$.
\end{itemize}

In our case, $S$ consists of certain invariants of diagrams that can be obtained by crossing changes from a fixed knot diagram $\cD$. An action is changing a crossing of $\cD$. A crossing change is deterministic, so $P_a(s, s')$ is 1 if a crossing change $a \in A_s$ is applied to the diagram associated to state $s$ and results in the state $s'$
and is 0 otherwise. 
If $a \in A_s$ leads to $s'$, then the reward $R_a(s, s')$ is $1$ if $s'$ is a diagram of the unknot $U$ and is $0$ otherwise.

The \emph{policy} $\pi$ is a potentially probabilistic mapping from $S$ to $A$. In state $s \in S$, the agent performs action $\pi(s) \in A_s$. The objective of training an RL agent is to choose $\pi$ to maximise the \emph{state value function} 
\[
V^\pi(s) := E \bigg( \sum_{t=0}^{\infty} \gamma^t R_{\pi(s_t)}(s_t, s_{t+1}) \bigg), 
\]
where $s_0 = s$, $s_{t+1} \sim P_{\pi(s_t)}(s_t, s_{t+1})$, and $\gamma \in [0,1]$ is called the \emph{discount factor}. This is the expected value of the total discounted reward the agent obtains using the policy $\pi$. The discount factor $\gamma$ determines how much weight is given to future rewards.

\subsection{Q-learning.} A classical approach to solving Markov decision processes is Q-learning, where `Q' stands for `quality'. Its goal is to learn the \emph{state-action value} $Q(s, a)$, which is the expected discounted  total reward if action $a \in A_s$ is taken from state $s \in S$, and the policy is followed thereafter. At time $t$, the agent selects action $a_t$, observes a reward $r_t$, and enters state $s_{t+1}$. We initialise $Q$ randomly and updated it via the \emph{Bellman equation}
\[
Q^{\text{new}}(s_t, a_t) := Q(s_t, a_t) + \alpha \left( r_t + \gamma \max_{a \in A_{s_{t+1}}} Q(s_{t+1}, a) - Q(s_t, a_t) \right),
\]
where $\alpha \in (0, 1]$ is the \emph{learning rate} or \emph{step size}. 

When selecting an action, we face the dilemma of \emph{exploration versus exploitation}; i.e., whether we explore the environment to potentially obtain a higher cumulative reward, or rely on the Q-values that we have learned so far. The $\varepsilon$-greedy policy blends the two approaches by performing a random action with probability $\varepsilon$ and an action $a_t \in A_{s_t}$ that maximises $Q(s_t, a_t)$ with probability $1 - \varepsilon$. 

A modern version of Q-learning is \emph{deep Q-learning}. Here, an artificial neural network $f \colon \R^S \to \R^A$ learns the Q-values, where  
\[
f(e_s) \cdot e_a = Q(s, a) 
\]
for the basis vector $e_s$ of $\R^S$ corresponding to the state $s \in S$ and the basis vector $e_a$ of $\R^A$ corresponding to the action $a \in A$. The weights of the network are updated using the Bellman equation and gradient descent.

\subsection{Importance weighted actor-learner architecture (IMPALA)} For the majority of our experiments, we used the IMPALA~\cite{IMPALA}  reinforcement learning architecture, which is a distributed agent developed for parallelisation. It learns the policy $\pi$ and the state value function $V^\pi$ via stochastic gradient ascent. Acting and learning are decoupled. A set of actors repeatedly generate trajectories of experience. One or more synchronous learners use this experience to learn the policy $\pi$. The policy the actors use lags behind the learners', which is corrected using a method called V-trace.

\subsection{Imitation learning} During imitation learning, an agent tries to learn a policy that mimics expert behaviour. It does not rely on a reward function. The simplest approach is \emph{behavioural cloning}, where a supervised learning model, usually a neural network, learns to map environment observations to (optimal) actions taken by an expert.

We mention two other, more sophisticated approaches to Imitation Learning. \emph{Adversarial imitation}, due to Ho and Ermon~\cite{Ho-Ermon}, is a minimax game between two AI models (Generative Adversarial Nets): the \emph{agent policy model} produces actions  using RL to attain the highest rewards from a \emph{reward model} that indicates how expert-like an action is, while the reward model attempts to distinguish the agent policy behaviour from expert behaviour. In the case of \emph{inverse Q-learning}, due to Garg et al.\,\cite{inverse-Q}, a single Q-function is learned. The policy is obtained by choosing the action with the highest Q value, and one can recover the reward from $Q$.

\section{Learning to unknot}

\subsection{Imitation learning and unknotting}\label{sec:imitation-unknotting} We used behavioural  cloning based on a NN to predict for each crossing the probability that it lies in a minimal unknotting set. If the predicted probability for a crossing $c$ is larger than $0.5$, then we interpret this such that $c$ does lie in a minimal unknotting set. Expert data was obtained from brute-forced minimal unknotting sets of randomly generated knots up to 30 crossings.

The distribution of the percentage of crossings that lie in a minimal unknotting set for diagrams in the dataset is shown in Figure~\ref{fig:fraction-valid-crossings}. The unknotting numbers of the diagrams range between 1 and 8.


Since knot diagrams are hard to feed into a neural network, the main features we used were invariants of the diagram, together with invariants of all diagrams obtained by changing one crossing. We call this \emph{one step lookahead}, which we also used in our RL agent. We computed invariants using SnapPy~\cite{SnapPy}.

We experimented with several collections of invariants of knot diagrams. These included the following:
\begin{enumerate}
	\item\label{it:basic} signature, determinant, writhe, and whether the diagram is alternating,
	\item\label{it:twist} number and lengths of twist regions, together with crossing signs,
	\item\label{it:Alex-Jones} Alexander polynomial $\Delta_K(t)$ and Jones polynomial $V_K(t)$ (coefficients, minimal and maximal degrees, and certain evaluations, including at some roots of unity, of the polynomial and its first three derivatives),
	\item\label{it:HFK} invariants computed by the \texttt{knot\char`_floer\char`_homology} function of SnapPy, written by Ozsv\'ath and Szab\'o ($\tau$, $\nu$, $\epsilon$, Seifert genus, is fibred, is L-space knot, rank in each homological grading, modulus, rank in each Alexander grading, total rank),
	\item\label{it:hyperbolic} whether the knot is hyperbolic, in which case we considered the volume, longitudinal translation, and natural slope.
\end{enumerate} 

Many invariants either failed to compute for a significant percentage of diagrams ($\ge 20\%$), such as the hyperbolic invariants in \eqref{it:hyperbolic}, or were slow to compute for large diagrams ($\ge 100$ crossings), such as the knot Floer homology invariants in \eqref{it:HFK}. For RL, it is important that the environment is fast. Hence, we focused on the simple invariants in \eqref{it:basic} and the polynomial features in \eqref{it:Alex-Jones}.

\begin{figure}
	\includegraphics[width=10cm]{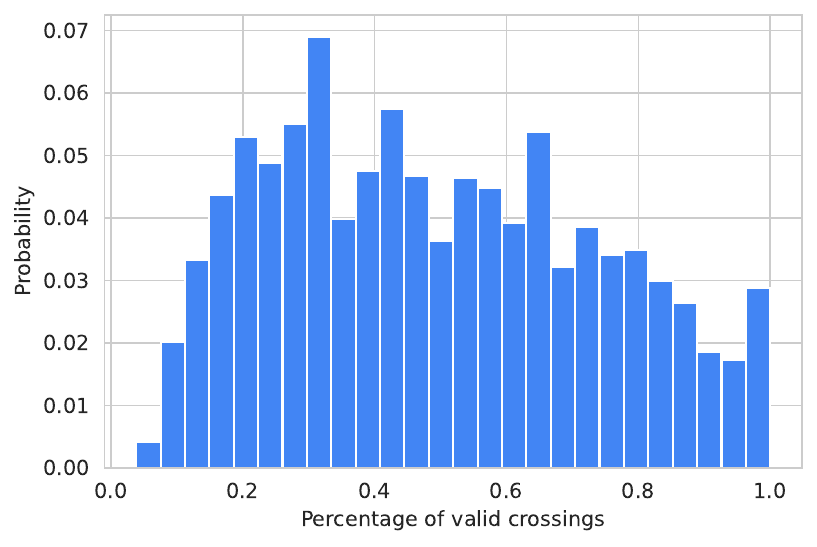}
	\caption{The distribution of the percentage of crossings that lie in a minimal unknotting set for diagrams in the dataset.}\label{fig:fraction-valid-crossings}
\end{figure}

To find out which subset of invariants to use, we set up a Supervised Learning experiment to predict whether a crossing lies in a minimal unknotting sequence. We considered neural networks with hidden layers of size [256, 256], [1024, 1024, 1024], and [2048, 2048, 2048], respectively. Learning was done in 10k steps, with learning rate 0.01 or 0.005, momentum 0.9, and batch size 2048. Accuracy ranged between 84.6\% and 88.1\%, with baseline 50.1\% when predicting the most commonly occurring crossing in unknotting sets (crossing 2 is in a minimal unknotting set for 50.1\% of diagrams in the dataset). Note that the indexing of the crossings is arbitrary in the dataset. The highest accuracy was achieved by hidden layers of size [1024,1024,1024] and learning rate 0.01 using all features. When only using the features in \eqref{it:Alex-Jones} on the same architecture, we obtained almost identical, 87.97\% accuracy.

Including the sum of the absolute values of the coefficients of the Alexander polynomial further improved the performance. Using both the Jones and the Alexander polynomials gave higher accuracy than the  Alexander polynomial only. This was also the case for the RL agent discussed in Section~\ref{sec:RL-unknotting}, and was particularly pronounced when it was forced to switch some inter-component crossings for connected sums (see Definition \ref{Def:inter-component} for the definition of an inter-component crossing).
Saliency analysis of the Alexander and Jones polynomial features revealed that the evaluations of $V_K(t)$ and $V_K'(t)$ near $1$ were most important in making the prediction.

One potential explanation of why the Jones polynomial works better than the Alexander polynomial is that the Jones polynomial is conjectured to detect the unknot, while the Alexander polynomial does not. The \emph{algebraic unknotting number} $u_a(K)$ of a knot $K$, due to Murakami~\cite{Murakami}, is the minimal number of crossing changes required to reach a knot with vanishing Alexander polynomial. Clearly, $u_a(K) \le u(K)$. Borodzik and Friedl~\cite{Borodzik-Friedl} have shown that $u_a(K)$ agrees with an invariant $n(K)$ that they defined using the Blanchfield form of $K$, and which can be computed in many examples. 
However, this in itself does not seem to completely explain why the Jones polynomial is the most useful feature for guiding the Imitation Learning and RL agents. Hence, it seems the Jones polynomial contains yet unobserved unknotting information.

\subsection{Reinforcement learning and unknotting}\label{sec:RL-unknotting} Our goal was to train an RL agent that performs crossing changes in a fixed diagram $\cD$ to unknot it, giving an \emph{upper bound} on $u(\cD)$. We used the IMPALA architecture. The resulting trained agent can determine $u(\cD)$ even when $c(\cD) \approx 200$, in which case brute-forcing is not possible. 

\begin{figure}
	\begin{center}
	\input{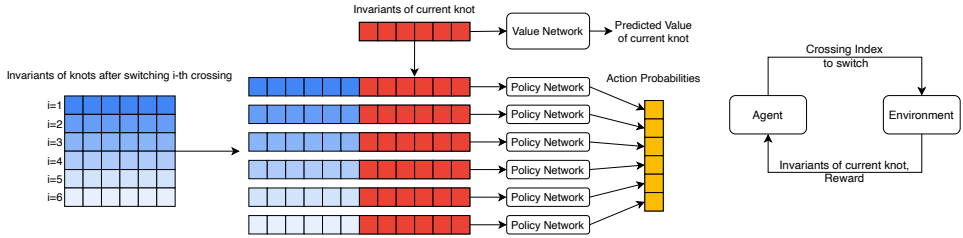}
	\end{center}
	\caption{The architecture of the RL agent. Input features are invariants of the current diagram, together with invariants of all neighbouring diagrams.} \label{fig:agent}
\end{figure}

The agent architecture is shown in Figure~\ref{fig:agent}.
As for imitation learning, the features we tried were invariants of the diagram (typically the Alexander and Jones polynomial features from \eqref{it:Alex-Jones} of Section~\ref{sec:imitation-unknotting}), together with invariants of all diagrams one can obtain via a single crossing change (one step lookahead). For each invariant, an additional boolean feature showed whether the invariant calculation had failed. Some invariants had a large range, especially Jones polynomial evaluations, which were clipped to lie in a fixed range. The agent was allowed to try a fixed number of crossing changes. It was trained on the dataset explained in Section~\ref{sec:additivity} consisting of knots with known unknotting numbers from KnotInfo, torus knots, quasipositive knots, and random knots with large signatures, and connected sums of these. We sampled randomly from these classes.

At the beginning of each episode, the environment samples a new knot and returns its invariants described above to the agent as a vector. Then, at each step, the agent computes an action; i.e.\ the crossing index that should be switched, by feeding the invariant vector through a neural network. The environment then performs the action by switching the crossing and returns the new invariant vector. An episode terminates if either the knot was unknotted successfully or if a maximum number of steps has been reached. If the unknot has been reached, a reward of $+1$ is returned, otherwise a reward of $0$.

Initially, we observed that the agent frequently revisited the same diagram. If a crossing appeared multiple times in an unknotting sequence, we counted it modulo 2. Furthermore, we experimented with disallowing revisiting the same Alexander or Jones polynomial, which decreased the percentage of unsolved diagrams for 50--100-crossing knots from 0.49\% to 0.05\% when using both the Alexander and the Jones polynomial. There was also a small increase in performance when disallowing Jones revisits only. When using an agent based on Alexander polynomial features only, disallowing Alexander polynomial revisits substantially improved performance, except when solving connected sums with forced inter-component crossings. When we allowed revisits, the agent sometimes found shorter unknotting sequences (after counting the number of times each crossing change was made modulo 2). The Alexander only policy was 10 times faster than the one using the Jones polynomial. Running the agent 10 times on each diagram substantially increased performance. 

When comparing the performance of the Alexander and the Jones polynomials on a dataset of random knots, the RL agent using the Jones polynomial both as features and for disallowing revisits found shorter unknotting sequences, though sometimes the agent based on the Alexander polynomial worked better; see Figure~\ref{fig:Alex_vs_Jones}. The percentage of knots that the RL agent could not unknot was much lower when using the Jones polynomial. 

We also compared the Jones unknotting agent to various baselines. A naive baseline of randomly switching crossings is fast and simple to implement, but unlikely to find minimal unknotting sequences for diagrams with many crossings. Similarly to~\cite{ribboning}, we developed a Bayesian Optimisation (BO) based random agent. The agent chooses randomly from the following action categories: Reidemeister I move, reverse Reidemeister I move, Reidemeister II move, reverse Reidemeister II move, Reidemeister III move, and crossing switch. The random agent chooses one of these action categories by randomly sampling according to certain weights. The concrete action is then sampled uniformly from all valid moves from that category, e.g.\ from all valid crossing indices. Each crossing change carries a reward of $-1$ and the reward is $0$ otherwise, which is different from the reward function we used for the RL agent. We optimised the sampling weights using BO using the average reward on a heldout validation set as the fitness metric. We limited the weights to be integers between $1$ and $100$. The best fitness was achieved by weight vectors where weights for  simplifying moves (Reidemeister I or II) as well as neutral moves (Reidemeister III) were near $100$, while weights for moves which increase the crossing count (reverse Reidemeister I and II) or which incur a reward penalty (crossing switch) were assigned a weight near $1$. Intuitively, this means the optimisation procedure converged to a solution which always attempts to maximally simplify the knot diagram before making a crossing switch. 
We then implemented further baselines which follow this insight of first simplifying then switching a crossing. \textit{Simplify (random)} utilizes SnapPy to first simplify the knot diagram before a random crossing is switched. \textit{Simplify (min.\ crossing)} always switches the crossing after which the resulting diagram has the least crossings after simplification. 
Figure~\ref{fig:unknotting_performance_comparison} shows how different unknotting strategies compare on a test set of $100,000$ diagrams. The RL agent utilizing the Jones polynomial manages to unknot more knots than the other strategies and is also capable of producing more efficient unknotting sequences.

\begin{figure}
	\begin{tikzpicture}  
		
		\begin{axis}  
			[  
			bar width = 1 cm,
			ybar,  
			enlargelimits=0.15,  
			ylabel={Number of diagrams},
			symbolic x coords={Jones better, equal, Alexander better}, 
			xtick=data,  
			nodes near coords,
			nodes near coords align={vertical},  
			]  
			\addplot coordinates {(Jones better,71699) (equal,29232) (Alexander better, 7069) };  
			
		\end{axis}  
	\end{tikzpicture}  
	\caption{Comparison of the performance of the Alexander and the Jones RL agents on a dataset of random knot diagrams.}
	\label{fig:Alex_vs_Jones}
\end{figure}

\begin{figure}
    \centering
    \begin{subfigure}[b]{0.48\textwidth}
    \caption{}
	\includegraphics[width=\textwidth]{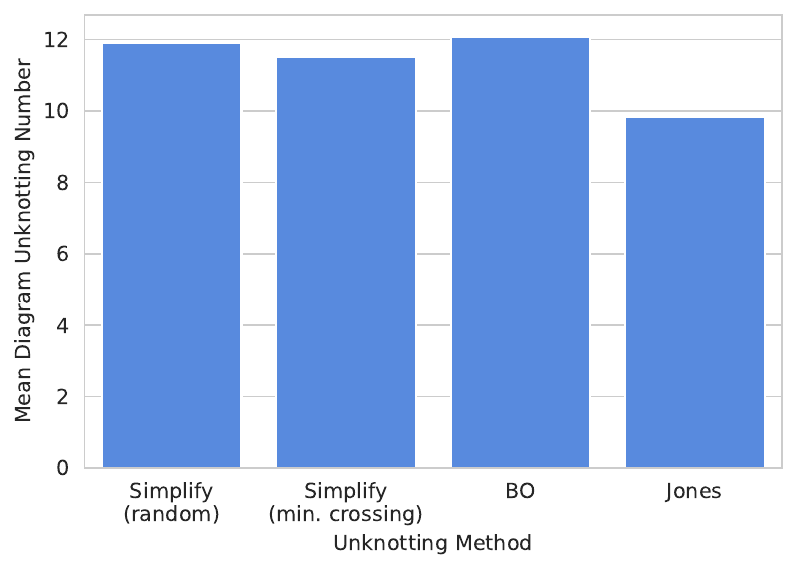}
	\end{subfigure}
	\begin{subfigure}[b]{0.48\textwidth}
	\caption{}
	\includegraphics[width=\textwidth]{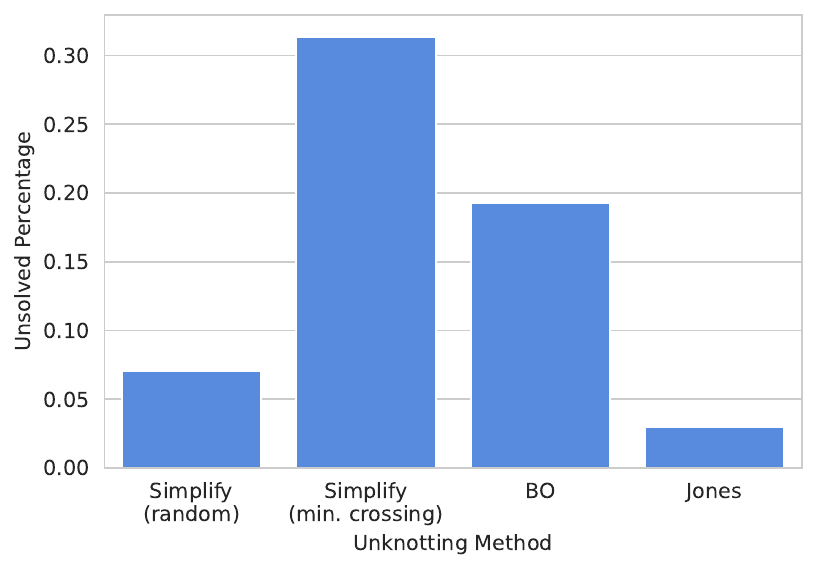}
	\end{subfigure}
	\caption{Comparison of the performance of different unknotting strategies. (A) shows the mean diagram unknotting number estimated by the different strategies for solved knots. (B) shows the percentage of knots each strategy was unable to unknot.}
	\label{fig:unknotting_performance_comparison}
\end{figure}

\subsection{Braids}\label{sec:braids} An alternative way to represent knots is as braid closures. This has a number of advantages. The Jones polynomial computation is exponential time. However, a polynomial-time algorithm exists for braid closures if we bound the braid index \cite{MortonShort}.
However, even this gets too slow for RL when the braid index is over 6--8. Hence, we considered braids of at most 8 strands. Furthermore, the slice--Bennequin inequality, due to Rudolph~\cite{Rudolph}, building on work of Kronheimer and Mrowka, provides easy-to-compute bounds on the unknotting number: 
\[
|w(\beta)| - n(\beta) + 1 \le 2u\big(\widehat{\beta}\big) \le c\big(\widehat{\beta}\big) + 1 - n(\beta),
\]
where $w(\beta)$ is the writhe and $n(\beta)$ is the number of strands of the braid word $\beta$, and $c(\widehat{\beta})$ is the crossing number of the braid closure $\widehat{\beta}$. We verified that our RL agent obtained statistically better bounds on the unknotting number than the slice--Bennequin bounds.

Let $\beta_1$ and $\beta_2$ be braids of indices $k_1$ and $k_2$, respectively. Furthermore, let $\sigma_{k_1}$ be the generator of the braid group of index $k_1 + k_2$ that swaps strands $k_1$ and $k_1 + 1$. Then $\widehat{\beta_1 \sigma_{k_1} \beta_2} = \widehat{\beta}_1 \# \widehat{\beta}_2$.
We further mixed $\widehat{\beta_1 \sigma_{k_1} \beta_2}$ by inserting braid words that are equivalent to the identity, and which mix strands between the two summands. We chose 7 such braid words to keep the size of the resulting braids manageable by the RL agent.
See, for example, Figure~\ref{fig:braid}. This method seemed to result in different mixing between the connected summands than overlaying them (see Section~\ref{sec:additivity} for a discussion on overlaying diagrams).

\begin{figure}
	\includegraphics[width=\textwidth, height=2cm]{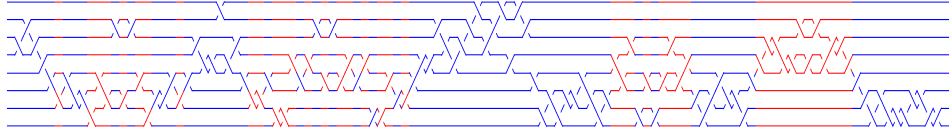}
	\caption{One can mix a diagram of a connected sum of two braid closures by inserting identity braid words. The inter-component crossings that yield hyperbolic knots are shown in red.}\label{fig:braid}
\end{figure}

We furthermore trained an RL agent which operated directly on braid words rather than using invariants. In this case, we used a transformer architecture~\cite{VaswaniShazeer}, an ML architecture designed to work on sequences. The input to the model is a sequence of integers representing the braid word, and the output is -- akin to the invariant-based agent -- a probability distribution over which crossing to switch. This had the big advantage of being invariant-free and hence very fast. It performed well on smaller braids ($\le 60$ crossings, 3--8 strands), but struggled to unknot larger braids efficiently. Future work could investigate invariant-free unknotting agents further.

A potential direction that we have not explored is to do Imitation Learning on the unknotting trajectories from the braid agent. One would filter trajectories that are close to minimal, and augment the dataset by rotating and mirroring the braid words, and by inserting braid identities.

\section{Additivity of the unknotting number}\label{sec:additivity}

We set out to search for a counterexample to the additivity of the unknotting number using our IMPALA agent. Our strategy was to first find a large dataset $S$ of knots with known unknotting numbers. We can assume that $\sigma(K) \ge 0$ for every $K \in S$ by mirroring it if $\sigma(K) < 0$. We then construct non-trivial diagrams of connected sums $K \# -K'$ for $K$, $K' \in S$. If the RL agent can unknot $K \# -K'$ using $u(K) + u(K') - 1$ crossing changes, then we are done.

\begin{figure}
	\includegraphics[width=0.49\textwidth]{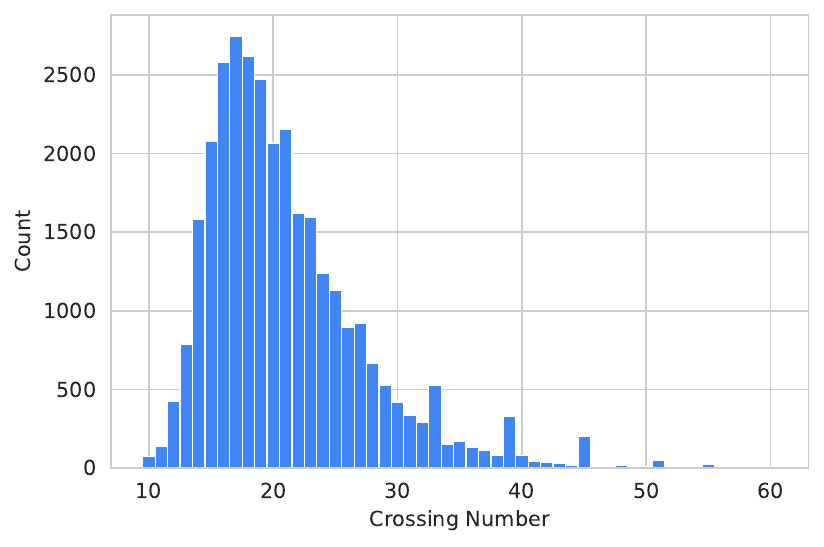}
	\includegraphics[width=0.49\textwidth]{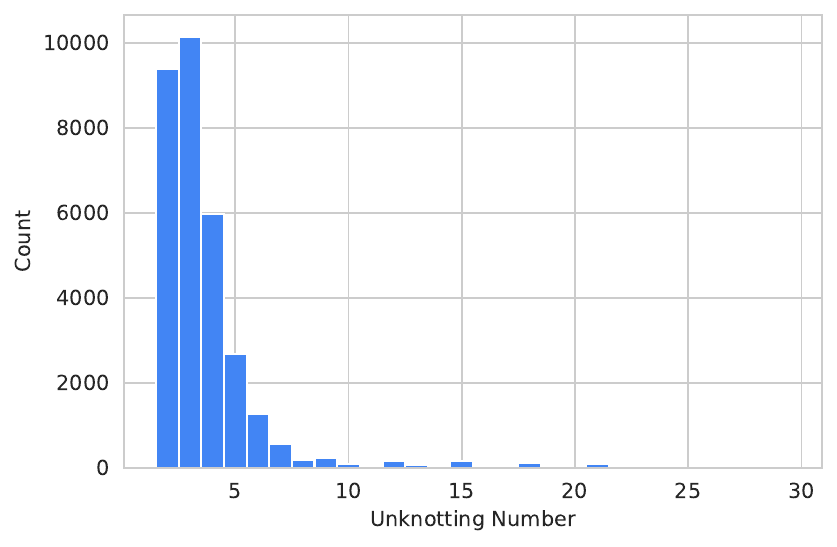}
	\includegraphics[width=0.49\textwidth]{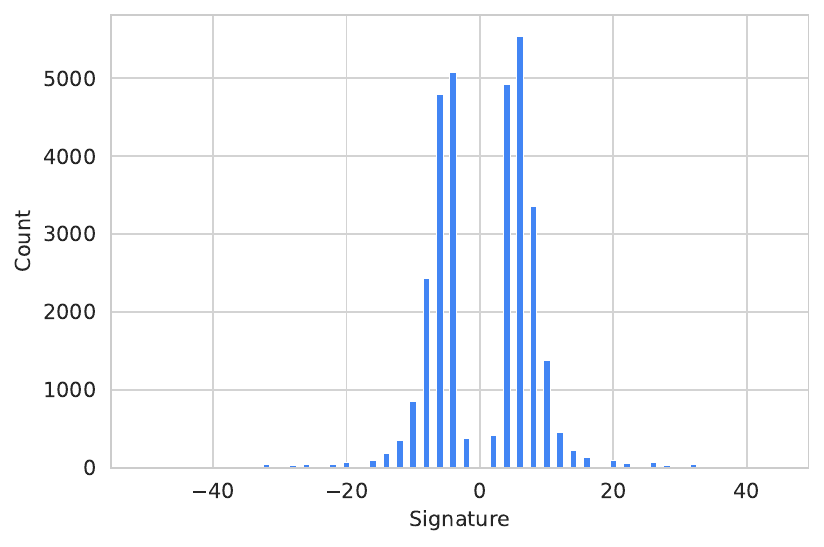}
	\caption{The distribution of random knots $K$ in our dataset by SnapPy simplified crossing number, by $u(K)$, and by $\sigma(K)$.}\label{fig:random-knots}
\end{figure}

To construct $S$, we obtained $u(K)$ for 31k random knots between 10 and 60 crossings and 26k quasipositive knots betwen 10 and 50 crossings in their SnapPy simplified diagram with large $|\sigma(K)|$, and where the upper bounds given by the RL agent and the lower bounds from knot Floer homology coincided.
We found it difficult to generate random knots $K$ with $|\sigma(K)|$ large. See Figure~\ref{fig:random-knots} for the distribution of the crossing number of a SnapPy simplified diagram and the unknotting numbers and signatures for the random knots in our dataset. For the lower bounds, we used the signature, the Ozsv\'ath--Szab\'o $\tau$ invariant, and $|\nu(\pm K)|$, which is sometimes 1 bigger than $|\tau(K)|$. (Note that $\sigma$ and $\tau$ are symmetric with respect to mirroring, while $\nu$ is not.) We also added torus knots and at most 12-crossing knots with known unknotting number from KnotInfo.

\begin{figure}
	\includegraphics[width=\textwidth]{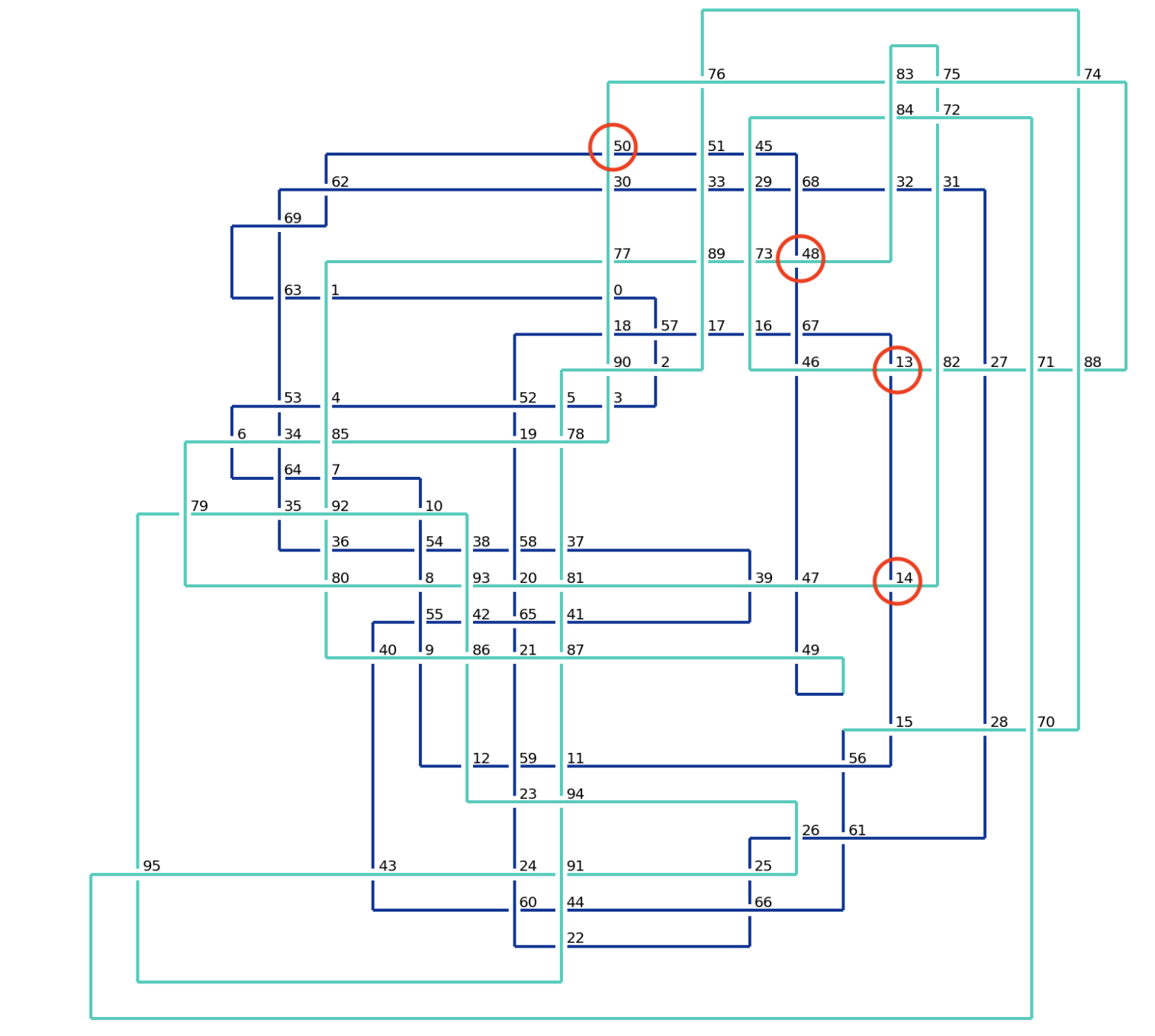}
	\caption{An overlaid diagram of a connected sum that can be unknotted by the inter-component crossings 13, 14, 48, and 50.}	\label{fig:strong-counterexample}
\end{figure}

In order to construct a non-trivial diagram of $K \# -K'$, we \emph{overlaid} the diagrams of $K$ and $K'$; see Figure~\ref{fig:strong-counterexample}. Note that, in an overlay sum, every crossing arc only intersects the connected sum sphere at most once.
Hence, in approximately 100k cases, we also performed \emph{random Reidemeister moves} to further mix the components. We did not always do this as the limit of our RL agent was around 200 crossings. For some experiments, we also considered connected overlay sums of 3 at most 12-crossing knots, with crossing number up to 120. In addition, we also searched among connected sums obtained from braids by inserting identity braid words, as explained in Section~\ref{sec:braids}. See Figure~\ref{fig:agent-performance} for the performance of the RL agent based on the Jones polynomial on connected sums, where the $x$-axis shows the sum of the unknotting numbers of the summands.

\begin{figure}
	\includegraphics[width=\textwidth]{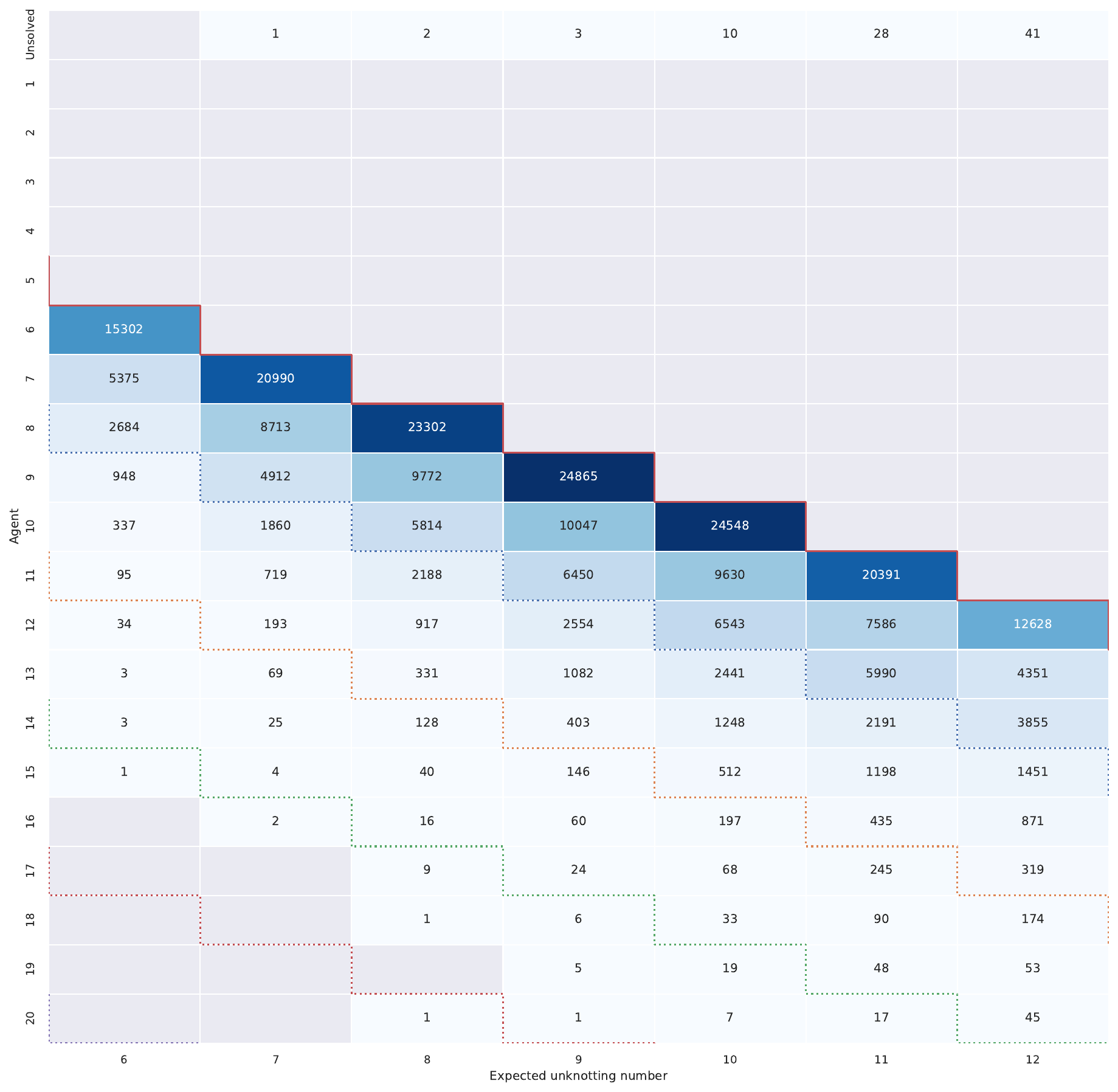}
	\caption{The performance of the RL agent based on the Jones polynomial on connected sums and no forced inter-component crossing changes. The four dashed lines indicate a difference of 3, 6, 9, and 12, respectively, between the diagram unknotting number found by the agent and the unknotting number expected assuming additivity.}\label{fig:agent-performance} 
\end{figure}

Instead of this stochastic approach to searching for a counterexample, one could train a different RL agent that searches for non-trivial unknotting crossing arcs in a fixed diagram, analogously to the approach of Gukov, Halverson, Manolescu, and Ruehle~\cite{ribboning} to finding ribbon disks. Note that a Bayesian random walker outperformed all their RL architectures, supporting the power of a stochastic approach.

Despite an extensive search, we have not found a counterexample to the additivity of the unknotting number. There are several potential interpretations of this. We could view it as some evidence supporting that the conjecture is true. Alternatively, a stochastic search might not be sufficient to come across a counterexample, if one exists. This is supported by the fact that it seems difficult to find an unknotting number 2 diagram of the knot 13n3370 featuring in the counterexample to the Bernhard--Jablan conjecture due to Brittenham and Hermiller~\cite{Brittenham-Hermiller}. A third explanation might be that, even though we have produced diagrams of connected sums with unknotting number at most $u(K) + u(K') - 1$, our RL agent was not good enough to find a minimal unknotting set. In the vast majority of cases, it did produce unknotting sets of size $u(K) + u(K')$, many of which even included inter-component crossings; see Figure~\ref{fig:strong-counterexample}. We will discuss this in more detail in the next section. 

If there is a counterexample, some crossings must be between the two components. Hence, we performed experiments where, in a dataset of 60 million knots, we forced the agent to change 2--5 inter-component crossings. For each knot, we applied the RL agent 10 times. In 25\% of the cases, it found unknotting sets of size $u(K) + u(K')$ containing the specified inter-component crossings. However, for 2\% of the diagrams, it did not find any unknotting set in the prescribed number of steps. This is not surprising, as the forced crossing changes might not be part of a minimal unknotting set. The above indicates that the RL agent is actually rather good at finding (close to) minimal unknotting sets in a given diagram.

\subsection{Strong conjecture}\label{sec:strong} While we have not found a counterexample to the additivity of the unknotting number, we have obtained counterexamples to the stronger form, Conjecture~\ref{conj:strong}. Recall that this states that, in every collection of unknotting crossing arcs for $K \# K'$, there is one that can be isotoped to be disjoint from the connected sum sphere. 

We now explain why Conjecture~\ref{conj:strong} implies the additivity of unknotting number (Conjecture~\ref{conj:additivity}). We must show that $u(K \# K') = u(K) + u(K')$ for knots $K$ and $K'$. We use induction on $u(K \# K')$. Let $\alpha_1, \dots, \alpha_n$ be a collection of crossing arcs for $K \# K'$ with $n = u(K \# K')$ and such that changing these crossings gives the unknot. Assuming Conjecture~\ref{conj:strong}, we can isotope these arcs so that one, $\alpha_1$ say, is disjoint from the 2-sphere specifying the connected sum. Say that it lies on the side of the 2-sphere corresponding to $K'$. Making the crossing change corresponding to $\alpha_1$ gives a connected sum $K \# K''$. Since $n$ was minimal, the remaining crossing arcs form a minimal unknotting sequence for $K \# K''$. So, $u(K \# K'') = n-1$, and therefore inductively, $u(K \# K'') = u(K) + u(K'')$. Now, $K''$ is obtained from $K'$ by changing a crossing, and therefore $u(K'') \geq u(K') - 1$. So,
\[
u(K \# K') = u(K \# K'') + 1 = u(K) + u(K'') + 1 \geq u(K) + u(K').
\]
The inequality $u(K \# K') \leq u(K) + u(K')$ holds trivially. Thus, we have shown that $u(K \# K') = u(K) + u(K')$, as required.

\begin{defn}
\label{Def:inter-component}
We say that a crossing change in a diagram of a connected sum is \emph{inter-component} if it results in a knot that is not a connected sum (e.g., hyperbolic) and is \emph{in-component} otherwise. 
\end{defn}

The motivation for this definition is as follows. As a result of Conjecture~\ref{conj:strong}, we are interested in crossing arcs that cannot be isotoped to be disjoint from the 2-sphere specifying the connected sum. However, this condition can be difficult to verify in practice. But, if a crossing arc for $K \# K'$ is inter-component in the above sense and $u(K)$, $u(K') > 1$, then it certainly cannot be isotoped to be disjoint from the connected sum 2-sphere.

We have found 20 counterexamples to Conjecture~\ref{conj:strong} by starting from diagrams where the RL agent found unknotting sets with several hyperbolic inter-component crossings and performing all in-component crossing changes from the unknotting set.
These produce connected sums where the remaining crossing changes result in a diagram of the unknot. In general, some of the inter-component crossings may become in-component after this procedure, but we have checked this is not the case for these 20 examples. One such example is shown in Figure~\ref{fig:strong-counterexample}, where the inter-component crossings 13, 14, 48, and 50 are unknotting.  

In many of our examples, changing the in-component crossings results in substantial simplification of the summands. By understanding the crossing arcs during this simplification led us to a more general method for constructing diagrams of connected sums that admit unknotting sets consisting of only inter-component crossings, which we recall from the introduction.

\strong*

The proof of this result will rely on Lickorish's work on prime tangles \cite{Lickorish-prime}. In fact, we will need to use a slight extension of his work, as follows.

\begin{defn}
A \emph{generalised tangle} is a pair $(B, t)$, where $B$ is a 3-ball and $t$ is 1-manifold properly embedded in $B$ that intersects $\partial B$ in $4$ points. So a generalised tangle consists of 2 arcs plus possibly some simple closed curves. When $t$ has no simple closed curve components, it is a \emph{tangle}. A tangle is \emph{trivial} when there is a homeomorphism from $B$ to $D^2 \times I$ taking $t$ to $\{ p_1, p_2 \} \times I$, where $p_1$ and $p_2$ are two points in the interior of $D^2$.
\end{defn}

The following definition is due to Lickorish~\cite{Lickorish-prime} in the case of tangles; we translate it verbatim to the setting of generalised tangles.

\begin{defn}
A generalised tangle $(B, t)$ is \emph{prime} if the following conditions both hold:
\begin{enumerate}
\item any 2-sphere in $B$ which meets $t$ transversely in $2$ points bounds a 3-ball intersecting $t$ in an unknotted arc;
\item it is not a trivial tangle.
\end{enumerate}
\end{defn}

The following result was proved by Lickorish~\cite{Lickorish-prime} for tangles. His proof extends immediately to generalised tangles, and is therefore omitted.

\begin{thm}
\label{Thm:LickorishGeneralised}
If two prime generalised tangles are glued via a homeomorphism between their boundary spheres that identifies the intersections with 1-manifolds, the result is a prime link.
\end{thm}

\begin{proof}[Proof of Theorem~\ref{thm:strong}]
	For $i \in \{1, 2\}$, isotope a sub-arc $a_i$ of $K_i$ such that it becomes straight and parallel to the $y$-axis in $\R^3$; see Figure~\ref{fig:crossings}. Furthermore, let $C_i$ be a set of disjoint unknotting framed crossing arcs of $K_i$ such that $|C_i| = u(K_i)$, and such that they are disjoint from $a_i$. Isotope the initial point $c(0)$ of each arc $c \in C_i$ into $a_i$. Then contract $c$ to a straight line segment via an ambient isotopy of $\R^3$ fixing $a_i$ such that $z(c(0)) > z(c(1))$, where $z \colon \R^3 \to \R$ is the $z$-coordinate function (this can be achieved by performing a finger move on $K_i$ by moving $c(1) \in K_i$ along $c$). We finally shorten and move the vertical arcs in $C_i$ so close to each other such that the corresponding crossings in the diagram $\cD_i$ obtained by perturbing $K_i$ and projecting it onto the $(x,y)$-plane are consecutive along the projection of $a_i$. Let $p_i \in a_i$ be a point such that $y(p_i) < y(c(0))$ for every $c \in C_i$. We take the connected sum of $K_1$ and $K_2$ at $p_1$ and $p_2$. We denote the resulting diagram of $K_1 \# K_2$ by $\cD$. See the upper left of Figure~\ref{fig:connected-sum}. In $K_1 \# K_2$, the arcs $a_1$ and $a_2$ will become arcs $a_+$ and $a_-$, where $a_+$ is the upper and $a_-$ the lower horizontal strand in the upper left of Figure~\ref{fig:connected-sum}. From now on, by a slight abuse of notation, $C_1$ and $C_2$ will denote sets of crossings rather than crossing arcs. Note that $a_+$ contains the crossings $C_1 \cup C_2$.
	
	\begin{figure}
\begingroup%
  \makeatletter%
  \providecommand\color[2][]{%
    \errmessage{(Inkscape) Color is used for the text in Inkscape, but the package 'color.sty' is not loaded}%
    \renewcommand\color[2][]{}%
  }%
  \providecommand\transparent[1]{%
    \errmessage{(Inkscape) Transparency is used (non-zero) for the text in Inkscape, but the package 'transparent.sty' is not loaded}%
    \renewcommand\transparent[1]{}%
  }%
  \providecommand\rotatebox[2]{#2}%
  \newcommand*\fsize{\dimexpr\f@size pt\relax}%
  \newcommand*\lineheight[1]{\fontsize{\fsize}{#1\fsize}\selectfont}%
  \ifx\svgwidth\undefined%
    \setlength{\unitlength}{136.56873117bp}%
    \ifx\svgscale\undefined%
      \relax%
    \else%
      \setlength{\unitlength}{\unitlength * \real{\svgscale}}%
    \fi%
  \else%
    \setlength{\unitlength}{\svgwidth}%
  \fi%
  \global\let\svgwidth\undefined%
  \global\let\svgscale\undefined%
  \makeatother%
  \begin{picture}(1,0.55742878)%
    \lineheight{1}%
    \setlength\tabcolsep{0pt}%
    \put(0,0){\includegraphics[width=\unitlength,page=1]{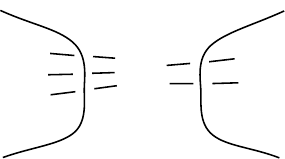}}%
    \put(0.73519893,0.1459649){\makebox(0,0)[lt]{\lineheight{1.25}\smash{\begin{tabular}[t]{l}$a_2$\end{tabular}}}}%
    \put(0.16261561,0.13749562){\makebox(0,0)[lt]{\lineheight{1.25}\smash{\begin{tabular}[t]{l}$a_1$\end{tabular}}}}%
    \put(0.09976661,0.51968856){\makebox(0,0)[lt]{\lineheight{1.25}\smash{\begin{tabular}[t]{l}$K_1$\end{tabular}}}}%
    \put(0.01024578,0.28419054){\makebox(0,0)[lt]{\lineheight{1.25}\smash{\begin{tabular}[t]{l}$C_1$\end{tabular}}}}%
    \put(0.86219413,0.29804342){\makebox(0,0)[lt]{\lineheight{1.25}\smash{\begin{tabular}[t]{l}$C_2$\end{tabular}}}}%
    \put(0.79471636,0.51507095){\makebox(0,0)[lt]{\lineheight{1.25}\smash{\begin{tabular}[t]{l}$K_2$\end{tabular}}}}%
    \put(0,0){\includegraphics[width=\unitlength,page=2]{crossings.pdf}}%
    \put(0.32833809,0.17105918){\makebox(0,0)[lt]{\lineheight{1.25}\smash{\begin{tabular}[t]{l}$p_1$\end{tabular}}}}%
    \put(0.59385062,0.1802944){\makebox(0,0)[lt]{\lineheight{1.25}\smash{\begin{tabular}[t]{l}$p_2$\end{tabular}}}}%
  \end{picture}%
\endgroup%

		\caption{The knots $K_i$ with the vertical arcs $a_i$, unknotting set of crossings $C_i$, and connected sum points $p_i$ for $i \in \{1, 2\}$.}
		\label{fig:crossings}
	\end{figure}
	
	\begin{figure}
		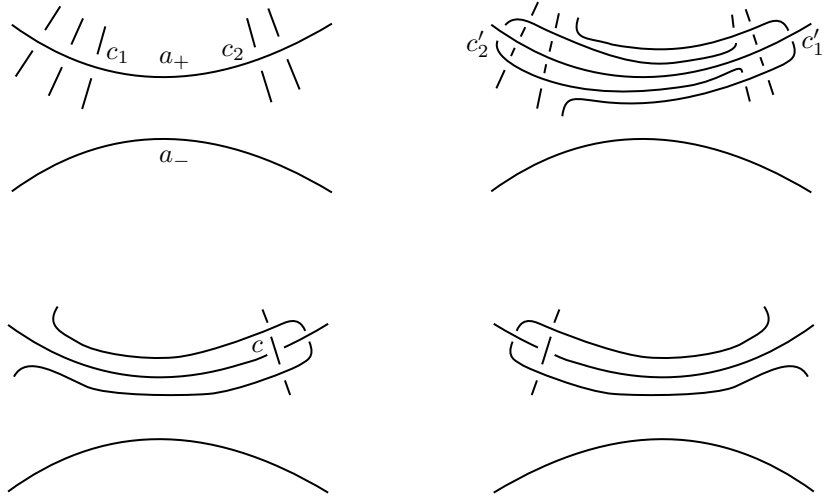
		\caption{Upper left: the diagram $\cD$ of the connected sum $K_1 \# K_2$. Upper right: the diagram $\cD'$ after performing finger moves on the crossings $c_1$ and $c_2$. Bottom left: the diagram near the connected sum sphere after changing one of the crossings in $C_2'$. Bottom right: after changing one of the crossings in $C_1'$. In the bottom row, there might be additional under-arcs across $a_+$ that are not shown.}
		\label{fig:connected-sum}
	\end{figure}
	
	\begin{figure}
\begingroup%
  \makeatletter%
  \providecommand\color[2][]{%
    \errmessage{(Inkscape) Color is used for the text in Inkscape, but the package 'color.sty' is not loaded}%
    \renewcommand\color[2][]{}%
  }%
  \providecommand\transparent[1]{%
    \errmessage{(Inkscape) Transparency is used (non-zero) for the text in Inkscape, but the package 'transparent.sty' is not loaded}%
    \renewcommand\transparent[1]{}%
  }%
  \providecommand\rotatebox[2]{#2}%
  \newcommand*\fsize{\dimexpr\f@size pt\relax}%
  \newcommand*\lineheight[1]{\fontsize{\fsize}{#1\fsize}\selectfont}%
  \ifx\svgwidth\undefined%
    \setlength{\unitlength}{137.33714254bp}%
    \ifx\svgscale\undefined%
      \relax%
    \else%
      \setlength{\unitlength}{\unitlength * \real{\svgscale}}%
    \fi%
  \else%
    \setlength{\unitlength}{\svgwidth}%
  \fi%
  \global\let\svgwidth\undefined%
  \global\let\svgscale\undefined%
  \makeatother%
  \begin{picture}(1,0.54746307)%
    \lineheight{1}%
    \setlength\tabcolsep{0pt}%
    \put(0,0){\includegraphics[width=\unitlength,page=1]{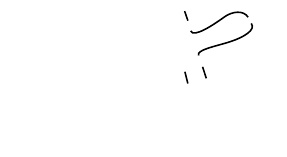}}%
    \put(0.81030651,0.35036995){\makebox(0,0)[lt]{\lineheight{1.25}\smash{\begin{tabular}[t]{l}$c_1'$\end{tabular}}}}%
    \put(0,0){\includegraphics[width=\unitlength,page=2]{d2.pdf}}%
    \put(0.87965233,0.50993401){\makebox(0,0)[lt]{\lineheight{1.25}\smash{\begin{tabular}[t]{l}$d_2'$\end{tabular}}}}%
    \put(0,0){\includegraphics[width=\unitlength,page=3]{d2.pdf}}%
  \end{picture}%
\endgroup%

		\caption{The result of changing the crossing $c_1'$.}
		\label{fig:d2}
	\end{figure}
    
	Suppose that $u(K_2) > 1$. The case $u(K_1) > 1$ is analogous. For $i \in \{1, 2\}$, let $c_i \in C_i$ be the crossing closest to $p_i$ and let $d_2$ be the crossing in $C_2$ furthest from $p_2$. Then perform a finger move on the lower strand at $c_1$ along $a_+$, isotoping it across the crossings in $C_2$. Next, isotope the lower strand at $c_2$ along $a_+$, isotoping it across the crossings in $C_1$. We write $c_i'$ for the crossing corresponding to $c_i$.  Perform a finger move on the lower strand at $d_2$ along $a_+$, moving it across $c_1'$, and resulting in the new crossing $d_2'$. The resulting diagram $\cD'$ is shown in the upper right of Figure~\ref{fig:connected-sum}. Changing all the crossings of $\cD$ and $\cD'$ along $a_+$ lead to equivalent diagrams, so, to the unknot. Let us write 
	\begin{equation}\label{eqn:crossings}
	C_1' := (C_1 \setminus \{c_1\}) \cup \{c_1'\} \text{ and } C_2' := (C_2 \setminus \{c_2, d_2\}) \cup \{c_2', d_2'\},
	\end{equation}
	and set $C := C_1' \cup C_2'$.
	
	If we change a single crossing in $C_2'$, after an isotopy, the diagram looks like the bottom left of Figure~\ref{fig:connected-sum} near the connected sum sphere. Analogously, changing a crossing in $C_1' \setminus \{c_1'\}$ leads to the diagram in the bottom right of Figure~\ref{fig:connected-sum}. Finally, changing $c_1'$, after an isotopy, results in the situation shown in Figure~\ref{fig:d2}.
	
	We claim that if we change any crossing $c \in C$, the resulting knot $K$ is prime. First, assume that $c \neq c_1'$. Without loss of generality, suppose that $c \in C_2'$, so, after an isotopy, we have the situation shown in the lower left of Figure~\ref{fig:connected-sum} (this is also the case for $c = c_2'$). The knot $K$ is split as a sum of two tangles $T_1$ and $T_2$ by the connected sum sphere $S$ of $K_1 \# K_2$, where $T_i$ lies on the same side of $S$ as $K_i$, for $i \in \{1, 2\}$.
	
	We now show that both $T_1$ and $T_2$ are prime. This will imply, by Lickorish's theorem \cite{Lickorish-prime} (see Theorem~\ref{Thm:LickorishGeneralised}), that their union $K$ is prime, as required. 
	
	We can glue a trivial tangle to $T_1$ to obtain $K_1$, which is prime by assumption. Hence, if $T_1$ were a trivial tangle, then $K_1$ would be a 2-bridge knot, which we have excluded by assumption. 
	
	We claim that $T_1$ contains no non-trivial ball-arc pair. Suppose that it did. We have already observed that we can attach a trivial tangle to the outside of $T_1$ to obtain $K_1$. Hence, if $T_1$ did contain a non-trivial ball-arc pair, this would be knotted like $K_1$ as $K_1$ is prime. 
	Let $K'_1$ be the result of changing the crossing $c_1$ of $K_1$. As $c_1$ is part of a minimal unknotting set for $K_1$, we have $u(K'_1) = u(K_1) - 1$, so $K_1' \neq K_1$.
	We note that we can also attach a trivial tangle in a different way to the outside of $T_1$ and obtain the knot $K'_1$. The non-trivial ball-arc pair in $T_1$ then forms a summand for $K'_1$. Hence, $K'_1 = K_1 \# J$ for some knot $J$, where $J \neq U$ since $K'_1 \neq K_1$. This contradicts the assumption that $K'_1$ is not the connected sum of $K_1$ and a non-trivial knot. Thus, we have shown that $T_1$ is a non-trivial tangle that contains no non-trivial ball-arc pair, and is hence prime.
	
	Let $K_2'$ be the knot obtained from $K_2$ by changing the crossing $c$. We can glue a trivial tangle to $T_2$ to obtain a link $L_2'$, which is $K_2'$ with an unknot linking the crossing $c$ non-trivially; see the left of Figure~\ref{fig:prime}. If $T_2$ were trivial, then $L_2'$ would be a 2-bridge link. Hence, $K_2'$ would be a 1-bridge knot; i.e., trivial, contradicting the assumption that $u(K_2) > 1$. 
	
	Our goal now is to show that $T_2$ has no non-trivial ball-arc pair. 
	The link $L'_2$ is the union of two tangles $A$ and $B$, where $A$ is shown on the right of Figure \ref{fig:prime}.
    
    \begin{figure}
        \centering
\begingroup%
  \makeatletter%
  \providecommand\color[2][]{%
    \errmessage{(Inkscape) Color is used for the text in Inkscape, but the package 'color.sty' is not loaded}%
    \renewcommand\color[2][]{}%
  }%
  \providecommand\transparent[1]{%
    \errmessage{(Inkscape) Transparency is used (non-zero) for the text in Inkscape, but the package 'transparent.sty' is not loaded}%
    \renewcommand\transparent[1]{}%
  }%
  \providecommand\rotatebox[2]{#2}%
  \newcommand*\fsize{\dimexpr\f@size pt\relax}%
  \newcommand*\lineheight[1]{\fontsize{\fsize}{#1\fsize}\selectfont}%
  \ifx\svgwidth\undefined%
    \setlength{\unitlength}{208.7875739bp}%
    \ifx\svgscale\undefined%
      \relax%
    \else%
      \setlength{\unitlength}{\unitlength * \real{\svgscale}}%
    \fi%
  \else%
    \setlength{\unitlength}{\svgwidth}%
  \fi%
  \global\let\svgwidth\undefined%
  \global\let\svgscale\undefined%
  \makeatother%
  \begin{picture}(1,0.35190129)%
    \lineheight{1}%
    \setlength\tabcolsep{0pt}%
    \put(0,0){\includegraphics[width=\unitlength,page=1]{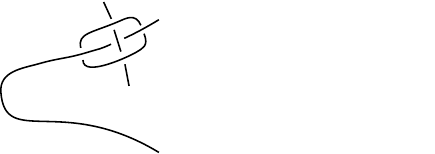}}%
    \put(0.22703214,0.25475472){\makebox(0,0)[lt]{\lineheight{1.25}\smash{\begin{tabular}[t]{l}$c$\end{tabular}}}}%
    \put(0.1080399,0.0088268){\makebox(0,0)[lt]{\lineheight{1.25}\smash{\begin{tabular}[t]{l}$L_2'$\end{tabular}}}}%
    \put(0,0){\includegraphics[width=\unitlength,page=2]{prime.pdf}}%
    \put(0.87498947,0.08585429){\makebox(0,0)[lt]{\lineheight{1.25}\smash{\begin{tabular}[t]{l}$A$\end{tabular}}}}%
    \put(0,0){\includegraphics[width=\unitlength,page=3]{prime.pdf}}%
  \end{picture}%
\endgroup%

        \caption{The left shows the link $L_2'$ obtained from $T_2$ by adding a trivial tangle. The right shows the tangle $A$.}
        \label{fig:prime}
    \end{figure}
    
    We claim that $A$ is prime. It is certainly not a trivial tangle, since it has a simple closed curve component. Suppose that there is a 2-sphere $S'$ intersecting the tangle $A$ transversely in two points that does not bound an unknotted arc. 
    Glue two copies of $A$ with a rotation through $\pi/2$. This has a prime non-split alternating diagram, and hence, by a theorem of Menasco~\cite{Menasco}, it is a prime non-split link. But $S'$ would force this link either to be split or not prime, which is a contradiction, proving the claim.

    We claim that $B$ is prime. Let $E'$ be the trivial tangle obtained from $A$ by deleting the unknot component. We obtain the trivial tangle $E$ from $E'$ by changing its only crossing~$c$.
    Then $E \cup B = K_2$. So, if $B$ were a trivial tangle, then $K_2$ would be a 2-bridge knot, which is contrary to our assumptions. Furthermore, if $B$ had a non-trivial ball-arc pair, then the arc would have to be knotted like $K_2$. But $E' \cup B = K_2'$,
    so $K_2' = K_2 \# J$ for some knot $J$. Since $c$ is part of a minimal unknotting set, $u(K_2') = u(K_2) - 1$, so $K_2' \neq K_2$, and hence $J \neq U$. However, we are assuming that this is not case. This proves that $B$ is indeed prime.

    Since we have proved that $A$ and $B$ are both prime generalised tangles, Theorem~\ref{Thm:LickorishGeneralised} implies that $L'_2 = A \cup B$ is a prime link.

    We are now in a position to show that $T_2$ has no non-trivial ball-arc pair. Suppose that, on the contrary, there is a sphere $S''$ in $T_2$ that intersects $T_2$ in two points
    and which bounds a non-trivial arc. This arc is therefore knotted like a non-trivial knot $K'$. Then 
    $S''$ transversely intersects $L_2'$ in two points,
    since $L'_2$ is the union of $T_2$ and a trivial tangle. On one side, $S''$ bounds an arc knotted like $K'$. On the other side, it also does not bound a trivial arc, since $L'_2$ is a link of two components. Hence, $L'_2$ is not prime, contrary to what we proved above.

    Thus, we have shown that $T_1$ and $T_2$ are both prime tangles, and therefore $K$ is prime, as required.
    
	We now consider the case when $c = c_1'$; see Figure~\ref{fig:d2}. As before, we split $K$ into tangles $T_1$ and $T_2$. The same argument shows that $T_1$ is prime. To see that $T_2$ is non-trivial, note that we can glue a trivial tangle to it to obtain a link $L_2$, which is isotopic to $K_2$ with an unknot linking the crossing $d_2$ non-trivially; see Figure~\ref{fig:weaving}. If $T_2$ were trivial, then $L_2$ would be a 2-bridge link. Hence, $K_2$ would be a trivial knot, contradicting the assumption that it is prime.  
	
	We can again write $L_2$ as a sum of tangles $A$ and $B$, where $A$ is shown in Figure~\ref{fig:prime}. We have already proved that $A$ is prime. Let $E$ be the trivial tangle obtained from $A$ by deleting the unknot component, and $E'$ the trivial tangle obtained by changing the crossing $d_2'$ of $E$. With this notation, the same argument as before shows that $B$ is also prime, and so $L_2 = A \cup B$ is prime. It follows that $T_2$ is prime and hence that $K = T_1 \cup T_2$ is prime. 
\end{proof}

\begin{rem}
We conjecture that Theorem~\ref{thm:strong} also holds without the extra assumption that $K_1$ and $K_2$ are not 2-bridge.
\end{rem}


\begin{figure}
    \centering
\begingroup%
  \makeatletter%
  \providecommand\color[2][]{%
    \errmessage{(Inkscape) Color is used for the text in Inkscape, but the package 'color.sty' is not loaded}%
    \renewcommand\color[2][]{}%
  }%
  \providecommand\transparent[1]{%
    \errmessage{(Inkscape) Transparency is used (non-zero) for the text in Inkscape, but the package 'transparent.sty' is not loaded}%
    \renewcommand\transparent[1]{}%
  }%
  \providecommand\rotatebox[2]{#2}%
  \newcommand*\fsize{\dimexpr\f@size pt\relax}%
  \newcommand*\lineheight[1]{\fontsize{\fsize}{#1\fsize}\selectfont}%
  \ifx\svgwidth\undefined%
    \setlength{\unitlength}{246.71305102bp}%
    \ifx\svgscale\undefined%
      \relax%
    \else%
      \setlength{\unitlength}{\unitlength * \real{\svgscale}}%
    \fi%
  \else%
    \setlength{\unitlength}{\svgwidth}%
  \fi%
  \global\let\svgwidth\undefined%
  \global\let\svgscale\undefined%
  \makeatother%
  \begin{picture}(1,0.25365674)%
    \lineheight{1}%
    \setlength\tabcolsep{0pt}%
    \put(0,0){\includegraphics[width=\unitlength,page=1]{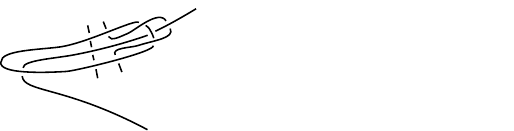}}%
    \put(0.34595479,0.17587037){\makebox(0,0)[lt]{\lineheight{1.25}\smash{\begin{tabular}[t]{l}$d_2'$\end{tabular}}}}%
    \put(0,0){\includegraphics[width=\unitlength,page=2]{weaving.pdf}}%
  \end{picture}%
\endgroup%

    \caption{The left shows the link $L_2$ obtained from the tangle $T_2$ by adding a trivial tangle. The right shows $L_2$ after the isotopy of the unknot component into a neighbourhood of the crossing $d_2$ of $K_2$.}
    \label{fig:weaving}
\end{figure}

\subsection{New unknotting numbers assuming additivity of $u$}
 If we assume that $u$ is additive and consider knots that appear along length $u(K) + u(K')$ unknotting trajectories of connected sums $K \# K'$ where $u(K)$ and $u(K')$ are both known, we obtain the unknotting number of 43 knots $K''$ with $c(K'') \le 12$ that are unknown; see Table~\ref{table:new}. More specifically, we considered knots that appear along minimal unknotting trajectories of counterexamples $\cD$ to Conjecture~\ref{conj:strong} by changing $u(\cD) - 4$ or $u(\cD) - 5$ crossings in a pre-computed unknotting set of $\cD$, then simplifying, obtaining a diagram $\cD'$, and then brute-forcing all unknotting trajectories from $\cD'$.
 
\begin{table}[h]
	\begin{center}
	\begin{tabular}{|c|c|c|c|c|c|c|c|c|}
		\hline
		12a824 & 12a835 & 12a878 & 12a898 & 12a916 & 12a981 & 12a999 & 12n80 & 12n71 \\
		\hline
		12n82 & 12n87 & 12n106 & 12n113 & 12n115 & 12n132 & 12n154 & 12n159 & 12n170\\
		\hline
		12n190 & 12n192 & 12n195 & 12n210 & 12n214 & 12n233 & 12n235 & 12n238 & 12n241 \\
		\hline
		12n246 & 12n309 & 12n315 & 12n346 & 12n437 & 12n500 & 12n548 & 12n670 & 12n673 \\
		\hline
		 12n675 & 12n678 & 12n681 & 12n690 & 12n695 & 12n721 & 12n723 & & \\
		 \hline
	\end{tabular}
	 \vspace{2mm}
	\end{center}
\caption{At most 12-crossings knots with unknown unknotting numbers that we found on minimal
	(assuming the additivity of the unknotting number) unknotting trajectories of connected sums $K \# K'$ where $u(K)$ and $u(K')$ are known.}\label{table:new}
\end{table}
 
 In all these examples, $u(K'')$ was equal to the KnotInfo upper bound. If one of these knots had lower unknotting number than the upper bound, it would imply that the unknotting number is not additive.
 
 \begin{figure}
 	\includegraphics{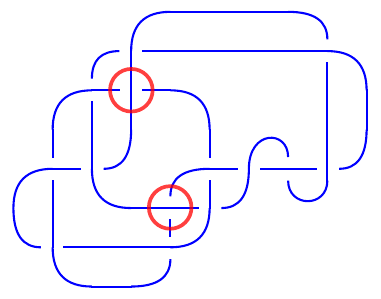}
 	\caption{A 14-crossing diagram of the knot 12a981. We can obtain $T_{2,5} \# -T_{2,7}$ by changing the two encircled crossings.} \label{fig:12a981}
 \end{figure}
 
Except for 12a898, 12a916, 12a981, and 12a999, all of these knots $K''$ have a crossing change in their KnotInfo diagram $\cD$ that results in a connected sum $K_0 \# K_1$ with $u(K_0)$ and $u(K_1)$ known and 
\[
u(\cD) = u(K_0) + u(K_1) - 1.
\]
Hence, for these knots, their unknotting number can be computed without the assistance of machine learning, assuming Conjecture~\ref{conj:additivity}. 
See Appendix~\ref{sec:trajectories} for conjecturally minimal unknotting sets of connected sums of knots with known unknotting numbers yielding 12a898, 12a916, and 12a999.

We have found by hand a 14-crossing diagram $\cD$ for 12a981 where two crossing changes yield a diagram of $T_{2,5} \# -T_{2,7}$;  see Figure~\ref{fig:12a981}. According to KnotInfo, $u(12a981) \in \{2, 3\}$. Assuming additivity of the unknotting number, 
\[
u(T_{2,3} \# -T_{2,7}) = u(T_{2,5}) + u(T_{2,7}) = 2 + 3 = 5, 
\]
which implies that $u(12a981) = 3$.
The diagram $\cD$ was obtained from $T_{2,5} \# -T_{2,7}$ by pushing a finger from one of the crossings of $T_{2,5}$ next to the connected sum point into $T_{2,7}$, then pushing a finger from the crossing of $T_{2,7}$ next to the connected sum away from the previous finger into $T_{2,5}$, and changing the other two crossings adjacent to the connected sum. This is the same procedure as in Section~\ref{sec:strong} that we used to change in-component to inter-component crossings; see Figure~\ref{fig:connected-sum}. Changing  the top left encircled crossing in Figure~\ref{fig:12a981} gives the hyperbolic knot 14n20178, while changing the lower right crossing results in the hyperbolic knot 14n20981.

\section{Hard unknot diagrams}

While the RL agent was running, it found 5,873,958 knot diagrams that had trivial Jones polynomial but which SnapPy could not simplify using the `level' algorithm. This performs random sequences of R3 moves, and simplifies using R1 and R2 moves whenever possible. As the algorithm is not deterministic, the check was repeated 25 times for each diagram. The knots were verified to be unknots by computing their Seifert genus via the \texttt{knot\char`_floer\char`_homology} function of SnapPy, written by Ozsv\'ath and Szab\'o. See Figure~\ref{fig:hard} for the distribution of the crossing number in the dataset.

\begin{figure}
	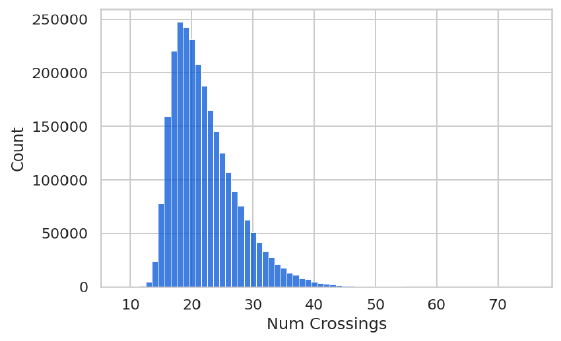
	\caption{The crossing number distribution of the hard unknot candidates.} \label{fig:hard}
\end{figure}

Note that SnapPy has a more sophisticated unknotting heuristic called `global', which, before performing `level' simplification, also attempts to perform \emph{pass moves} that decrease the number of crossings. These consist of picking up a strand that runs over or under the rest of the diagram (corresponding to a sequence of consecutive over- or under-crossings) and putting it down somewhere else. It is possible, however, that a diagram can be simplified using R1 and R2 moves following a sequence of pass moves that do not change the crossing number, and the `global' heuristic would not find such a simplification.
Petronio and Zanellati~\cite{Petronio-Zanellati} have given an example of a 120-crossing hard unknot diagram that cannot be monotonically simplified using Reidemeister moves and pass moves.

\begin{figure}
	\begin{tikzpicture}  
		
		\begin{axis}  
			[
			scale only axis,
			width = 11 cm,
			height = 6 cm,
			bar width = 1.5 cm,
			ybar,  
			enlargelimits=0.15,  
			ylabel={Number of diagrams},
			symbolic x coords={Total, Hard, Distinct hard, Timed out, Not hard}, 
			xtick=data,  
			nodes near coords,
			nodes near coords align={vertical},  
			]  
			\addplot coordinates {(Total, 5873958) (Hard, 5364424) (Distinct hard, 2623203) (Timed out, 101339) (Not hard, 408195)};  
			
		\end{axis}  
	\end{tikzpicture}  
	\caption{The dataset of hard unknot candidates.}
	\label{fig:hard_dataset}
\end{figure}
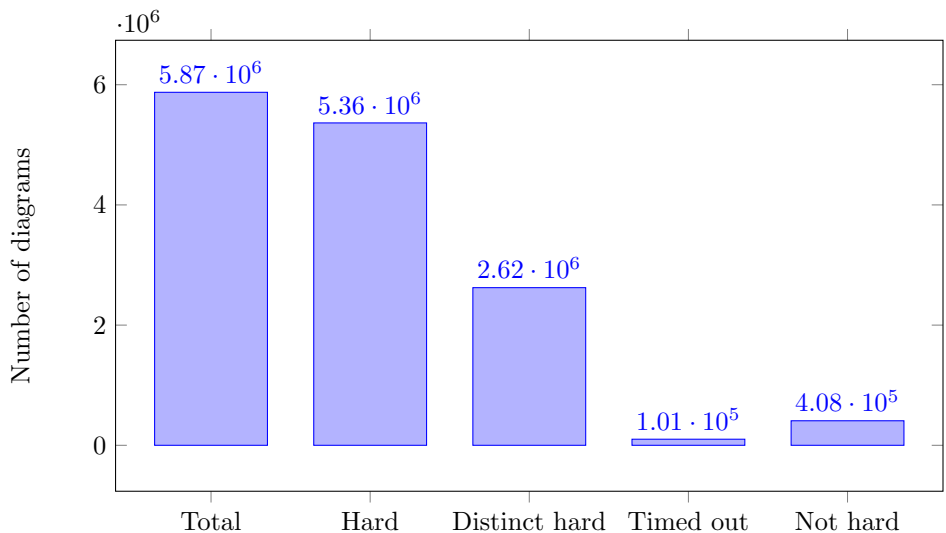

We verified that 5,364,424 of these diagrams are hard; see Figure~\ref{fig:hard_dataset}. We did this by listing all possible diagrams that can be obtained using R3 moves, and showing that none of these can be simplified using R1 and R2 moves. Typically, the number of R3-equivalent diagrams is small, but in some cases there are more than $10^4$. Hence, we set a 2-minute timeout. Our algorithm only timed out in 101,339 cases, and showed that the diagram was not hard in 408,195 cases. By comparing the Gauss codes of the remaining 5,364,424 hard diagrams, we obtained 2,623,203 distinct hard unknot diagrams. These fall into 2,464,461 R3-equivalence classes, where no two diagrams are related by a sequence of R3 moves. We are making these diagrams publicy available alongside this paper, see Appendix \ref{ref:data}.

We attempted to find `really hard' unknot diagrams which SnapPy fails to simplify completely even with its `global' heuristic. We checked for each of the above hard unknot diagram whether it can be reduced to the trivial diagram by either $10^6$ independent simplification attempts, where we ran the simplification algorithm repeatedly on the same diagram, or $10^5$ subsequent simplification attempts, where the output of each simplification is used as the input of the next attempt. We identified 382 hard unknot diagram that could not be simplified with either of these two methods. These diagrams are being made available; see Appendix~\ref{ref:data}.

An example of a hard unknot diagram with 6225 R3-equivalent diagrams is shown in Figure~\ref{fig:large-moduli}, which survived $10^4$ subsequent simplification attempts. Its PD code is 

\vspace{2mm}
\noindent {\small [[62, 25, 63, 26], [59, 4, 60, 5], [66, 17, 67, 18], [63, 11, 64, 10], [73, 19, 74, 18], [83, 33, 84, 32], [78, 41, 79, 42], [38, 31, 39, 32], [52, 46, 53, 45], [75, 46, 76, 47], [44, 52, 45, 51], [65, 55, 66, 54], [79, 48, 80, 49], [74, 53, 75, 54], [40, 57, 41, 58], [70, 61, 71, 62], [60, 69, 61, 70], [49, 76, 50, 77], [68, 71, 69, 72], [47, 80, 48, 81], [81, 57, 82, 56], [42, 77, 43, 78], [50, 44, 51, 43], [11, 16, 12, 17], [5, 26, 6, 27], [7, 28, 8, 29], [34, 29, 35, 30], [8, 36, 9, 35], [30, 37, 31, 38], [27, 6, 28, 7], [36, 10, 37, 9], [15, 12, 16, 13], [23, 15, 24, 14], [3, 21, 4, 20], [67, 23, 68, 22], [72, 21, 73, 22], [13, 25, 14, 24], [55, 65, 56, 64], [33, 1, 34, 84], [58, 2, 59, 1], [39, 82, 40, 83], [19, 3, 20, 2]]}
\vspace{2mm}

\begin{figure}
	\includegraphics[width = 10cm]{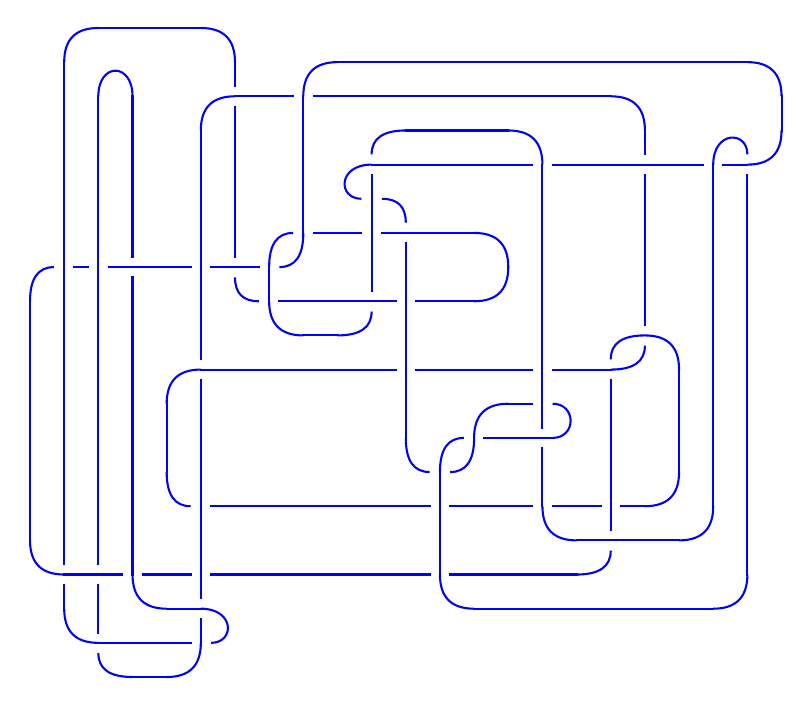}
	\caption{A 42-crossing hard unknot diagram with 6225 R3-equivalent diagrams that we have not been able to simplify by calling SnapPy's `global' heuristic 10000 times.}
	\label{fig:large-moduli}
\end{figure}

\bibliographystyle{plain}
\bibliography{topology}

\begin{thebibliography}{10}

\bibitem{Alishahi-Eftekhary-unknotting}
Akram Alishahi and Eaman Eftekhary.
\newblock Knot {F}loer homology and the unknotting number.
\newblock {\em Geom. Topol.}, 24(5):2435--2469, 2020.

\bibitem{Baader}
S.~Baader.
\newblock Note on crossing changes.
\newblock {\em Q. J. Math.}, 57(2):139--142, 2006.

\bibitem{Bernhard}
James~A. Bernhard.
\newblock Unknotting numbers and minimal knot diagrams.
\newblock {\em J. Knot Theory Ramifications}, 3(1):1--5, 1994.

\bibitem{Borodzik-Friedl}
Maciej Borodzik and Stefan Friedl.
\newblock On the algebraic unknotting number.
\newblock {\em Trans. London Math. Soc.}, 1(1):57--84, 2014.

\bibitem{Brittenham-Hermiller}
Mark Brittenham and Susan Hermiller.
\newblock A counterexample to the {B}ernhard-{J}ablan unknotting conjecture.
\newblock {\em Exp. Math.}, 30(4):547--556, 2021.

\bibitem{Burde-Zieschang}
Gerhard Burde, Heiner Zieschang, and Michael Heusener.
\newblock {\em Knots}, volume~5 of {\em De Gruyter Studies in Mathematics}.
\newblock De Gruyter, Berlin, extended edition, 2014.

\bibitem{hard}
Benjamin~A. Burton, Hsien-Chih Chang, Maarten Löffler, Clément Maria, Arnaud
  de~Mesmay, Saul Schleimer, Eric Sedgwick, and Jonathan Spreer.
\newblock Hard diagrams of the unknot.
\newblock {\em Experimental Mathematics}, 0(0):1--19, 2023.

\bibitem{SnapPy}
Marc Culler, Nathan~M. Dunfield, Matthias Goerner, and Jeffrey~R. Weeks.
\newblock Snap{P}y, a computer program for studying the geometry and topology
  of $3$-manifolds.
\newblock Available at \url{http://.computop.org} (12/04/2023).

\bibitem{Nature}
A.~Davies, P.~Veli\v{c}kovi\'c, L.~Buesing, S.~Blackwell, D.~Zheng,
  N.~Toma\v{s}ev, R.~Tanburn, P.~Battaglia, C.~Blundell, A.~Juh\'asz,
  M.~Lackenby, G.~Williamson, D.~Hassabis, and P.~Kohli.
\newblock \href{https://doi.org/10.1038/s41586-021-04086-x}{Advancing
  mathematics by guiding human intuition with {AI}}.
\newblock {\em Nature}, 600:70--74, 2021.

\bibitem{IMPALA}
Lasse Espeholt, Hubert Soyer, Remi Munos, Karen Simonyan, Vlad Mnih, Tom Ward,
  Yotam Doron, Vlad Firoiu, Tim Harley, Iain Dunning, Shane Legg, and Koray
  Kavukcuoglu.
\newblock {IMPALA}: Scalable distributed deep-{RL} with importance weighted
  actor-learner architectures.
\newblock In Jennifer Dy and Andreas Krause, editors, {\em Proceedings of the
  35th International Conference on Machine Learning}, volume~80 of {\em
  Proceedings of Machine Learning Research}, pages 1407--1416. PMLR, 10--15 Jul
  2018.

\bibitem{inverse-Q}
Divyansh Garg, Shuvam Chakraborty, Chris Cundy, Jiaming Song, and Stefano
  Ermon.
\newblock Iq-learn: Inverse soft-q learning for imitation.
\newblock In M.~Ranzato, A.~Beygelzimer, Y.~Dauphin, P.S. Liang, and J.~Wortman
  Vaughan, editors, {\em Advances in Neural Information Processing Systems},
  volume~34, pages 4028--4039. Curran Associates, Inc., 2021.

\bibitem{ribboning}
Sergei Gukov, James Halverson, Ciprian Manolescu, and Fabian Ruehle.
\newblock Searching for ribbon knots with machine learning.
\newblock 2023.
\newblock \url{https://arxiv.org/abs/2304.09304}.

\bibitem{unknotting}
Sergei Gukov, James Halverson, Fabian Ruehle, and Piotr Su{\l}kowski.
\newblock \href{https://doi.org/10.1088/2632-2153/abe91f}{Learning to unknot}.
\newblock {\em Mach. Learn.: Sci. Technol.}, 2(2):025035, 2021.

\bibitem{Ho-Ermon}
Jonathan Ho and Stefano Ermon.
\newblock Generative adversarial imitation learning.
\newblock In D.~Lee, M.~Sugiyama, U.~Luxburg, I.~Guyon, and R.~Garnett,
  editors, {\em Advances in Neural Information Processing Systems}, volume~29.
  Curran Associates, Inc., 2016.

\bibitem{universal-approximation}
Kurt Hornik, Maxwell Stinchcombe, and Halbert White.
\newblock Multilayer feedforward networks are universal approximators.
\newblock {\em Neural Networks}, 2(5):359--366, 1989.

\bibitem{Hughes}
Mark~C. Hughes.
\newblock \href{https://doi.org/10.1142/S0218216520500054}{A neural network
  approach to predicting and computing knot invariants}.
\newblock {\em J. Knot Theory Ramifications}, 29(3):2050005, 20, 2020.

\bibitem{Jablan}
Slavik~V. Jablan.
\newblock Unknotting number and {$\infty$}-unknotting number of a knot.
\newblock {\em Filomat}, (12, part 1):113--120, 1998.

\bibitem{Juhasz-book}
Andr\'as Juh\'asz.
\newblock {\em Differential and low-dimensional topology}, volume 104 of {\em
  London Mathematical Society Student Texts}.
\newblock Cambridge University Press, Cambridge, [2023] \copyright2023.

\bibitem{Khovanov-homology}
Mikhail Khovanov.
\newblock A categorification of the {J}ones polynomial.
\newblock {\em Duke Math. J.}, 101(3):359--426, 2000.

\bibitem{Kronheimer-Mrowka-Milnor}
P.~B. Kronheimer and T.~S. Mrowka.
\newblock Gauge theory for embedded surfaces. {I}.
\newblock {\em Topology}, 32(4):773--826, 1993.

\bibitem{Lackenby-survey}
Marc Lackenby.
\newblock Elementary knot theory.
\newblock In {\em Lectures on geometry}, Clay Lect. Notes, pages 29--64. Oxford
  Univ. Press, Oxford, 2017.

\bibitem{Lickorish-prime}
W.~B.~Raymond Lickorish.
\newblock Prime knots and tangles.
\newblock {\em Trans. Amer. Math. Soc.}, 267(1):321--332, 1981.

\bibitem{Lickorish-book}
William B.~R. Lickorish.
\newblock {\em An introduction to knot theory}, volume 175 of {\em Graduate
  Texts in Mathematics}.
\newblock Springer-Verlag, New York, 1997.

\bibitem{knotinfo}
Charles Livingston and Allison~H. Moore.
\newblock Knotinfo: Table of knot invariants.
\newblock URL: \url{knotinfo.math.indiana.edu}, April 2023.

\bibitem{Menasco}
W.~Menasco.
\newblock Closed incompressible surfaces in alternating knot and link
  complements.
\newblock {\em Topology}, 23(1):37--44, 1984.

\bibitem{MortonShort}
H.~R. Morton and H.~B. Short.
\newblock Calculating the {$2$}-variable polynomial for knots presented as
  closed braids.
\newblock {\em J. Algorithms}, 11(1):117--131, 1990.

\bibitem{Murakami}
Hitoshi Murakami.
\newblock Algebraic unknotting operation.
\newblock In {\em Proceedings of the {S}econd {S}oviet-{J}apan {J}oint
  {S}ymposium of {T}opology ({K}habarovsk, 1989)}, volume~8, pages 283--292,
  1990.

\bibitem{OSz-4-ball}
Peter Ozsv\'{a}th and Zolt\'{a}n Szab\'{o}.
\newblock Knot {F}loer homology and the four-ball genus.
\newblock {\em Geom. Topol.}, 7:615--639, 2003.

\bibitem{OSz-HFK}
Peter Ozsv\'{a}th and Zolt\'{a}n Szab\'{o}.
\newblock Holomorphic disks and knot invariants.
\newblock {\em Adv. Math.}, 186(1):58--116, 2004.

\bibitem{Petronio-Zanellati}
Carlo Petronio and Adolfo Zanellati.
\newblock Algorithmic simplification of knot diagrams: new moves and
  experiments.
\newblock {\em J. Knot Theory Ramifications}, 25(10):1650059, 30, 2016.

\bibitem{Rasmussen-s}
Jacob Rasmussen.
\newblock Khovanov homology and the slice genus.
\newblock {\em Invent. Math.}, 182(2):419--447, 2010.

\bibitem{Ras}
Jacob~A. Rasmussen.
\newblock Floer homology and knot complements.
\newblock {\em PhD thesis, Harvard University}, 2003.

\bibitem{Rolfsen-book}
Dale Rolfsen.
\newblock {\em Knots and links}, volume~7 of {\em Mathematics Lecture Series}.
\newblock Publish or Perish, Inc., Houston, TX, 1990.
\newblock Corrected reprint of the 1976 original.

\bibitem{Rudolph}
Lee Rudolph.
\newblock Quasipositivity as an obstruction to sliceness.
\newblock {\em Bull. Amer. Math. Soc. (N.S.)}, 29(1):51--59, 1993.

\bibitem{Scharlemann-unknotting}
Martin~G. Scharlemann.
\newblock Unknotting number one knots are prime.
\newblock {\em Invent. Math.}, 82(1):37--55, 1985.

\bibitem{Taniyama}
Kouki Taniyama.
\newblock Unknotting numbers of diagrams of a given nontrivial knot are
  unbounded.
\newblock {\em J. Knot Theory Ramifications}, 18(8):1049--1063, 2009.

\bibitem{VaswaniShazeer}
Ashish Vaswani, Noam Shazeer, Niki Parmar, Jakob Uszkoreit, Llion Jones,
  Aidan~N. Gomez, Lukasz Kaiser, and Illia Polosukhin.
\newblock Attention is all you need.
\newblock {\em CoRR}, abs/1706.03762, 2017.

\bibitem{Welsh}
D.~J.~A. Welsh.
\newblock On the number of knots and links.
\newblock In {\em Sets, graphs and numbers ({B}udapest, 1991)}, volume~60 of
  {\em Colloq. Math. Soc. J\'anos Bolyai}, pages 713--718. North-Holland,
  Amsterdam, 1992.

\end{thebibliography}

\appendix
\section{Data Availability}
\label{ref:data}
We are making the following datasets available alongside this paper at \href{https://storage.mtls.cloud.google.com/gdm-unknotting}{gs://gdm-unknotting}.

\begin{enumerate}
    \item Knots and their unknotting numbers: 31,380 random knots and 26,473 quasipositive knots.
    \item 2,464,461 hard unknot diagrams.
    \item 382 `really hard' unknot diagrams.
    \item 20 counterexamples to Conjecture~\ref{conj:strong}.
\end{enumerate}

\section{Unknotting trajectories}\label{sec:trajectories}

Here, we give unkotting trajectories of length $u(K_1) + u(K_2)$ of diagrams of connected sums $K_1 \# K_2$, where $K_1$ and $K_2$ have known unknotting numbers, and the trajectories pass through the knots 12a898, 12a916, and 12a999 with currently unknown unknotting numbers, respectively. Note that, for 12a916 and 12a999, we switched 5 crossings of an unknotting set of size 9 found using RL, then simplified the resulting diagram, and switched one further crossing along a minimal unknotting set found using brute force.

\subsection{12a898}

Initial PD code of $K_1 \# K_2$: 

\noindent [[135, 86, 136, 87], [133, 11, 134, 10], [128, 78, 129, 77], [131, 70, 132, 71], [129, 65, 130, 64], [132, 57, 133, 58], [125, 76, 126, 77], [124, 76, 125, 75], [122, 186, 123, 185], [117, 43, 118, 42], [119, 21, 120, 20], [121, 18, 122, 19], [116, 191, 117, 192], [112, 47, 113, 48], [114, 23, 115, 24], [111, 17, 112, 16], [107, 62, 108, 63], [105, 61, 106, 60], [103, 8, 104, 9], [101, 183, 102, 182], [99, 3, 100, 2], [98, 38, 99, 37], [97, 28, 98, 29], [172, 188, 173, 187], [94, 89, 95, 90], [175, 81, 176, 80], [180, 55, 181, 56], [176, 50, 177, 49], [173, 46, 174, 47], [170, 20, 171, 19], [96, 193, 97, 0], [168, 189, 169, 190], [165, 52, 166, 53], [167, 45, 168, 44], [163, 84, 164, 85], [164, 25, 165, 26], [161, 68, 162, 69], [158, 84, 159, 83], [160, 68, 161, 67], [146, 184, 147, 183], [157, 82, 158, 83], [152, 74, 153, 73], [155, 66, 156, 67], [150, 62, 151, 61], [154, 14, 155, 13], [148, 9, 149, 10], [144, 36, 145, 35], [143, 30, 144, 31], [141, 88, 142, 89], [138, 33, 139, 34], [139, 33, 140, 32], [136, 88, 137, 87], [153, 130, 154, 131], [110, 127, 111, 128], [108, 123, 109, 124], [171, 120, 172, 121], [166, 115, 167, 116], [126, 109, 127, 110], [149, 104, 150, 105], [145, 100, 146, 101], [140, 93, 141, 94], [134, 181, 135, 182], [162, 179, 163, 180], [156, 177, 157, 178], [113, 174, 114, 175], [118, 169, 119, 170], [178, 159, 179, 160], [106, 151, 107, 152], [102, 147, 103, 148], [95, 142, 96, 143], [92, 137, 93, 138], [43, 191, 44, 190], [21, 189, 22, 188], [7, 185, 8, 184], [31, 91, 32, 90], [54, 86, 55, 85], [48, 80, 49, 79], [15, 78, 16, 79], [63, 75, 64, 74], [58, 72, 59, 71], [12, 70, 13, 69], [72, 60, 73, 59], [26, 53, 27, 54], [81, 51, 82, 50], [6, 42, 7, 41], [4, 40, 5, 39], [1, 37, 2, 36], [91, 35, 92, 34], [192, 28, 193, 27], [51, 24, 52, 25], [45, 23, 46, 22], [186, 18, 187, 17], [65, 15, 66, 14], [56, 11, 57, 12], [40, 6, 41, 5], [38, 4, 39, 3], [29, 1, 30, 0]] 

\noindent PD code of first summand $K_1$: 

\noindent [(3, 18, 4, 19), (29, 10, 30, 11), (31, 12, 32, 13), (33, 14, 34, 15), (24, 17, 25, 18), (6, 37, 7, 0), (7, 26, 8, 27), (16, 23, 17, 24), (1, 20, 2, 21), (21, 2, 22, 3), (27, 8, 28, 9), (22, 15, 23, 16), (25, 34, 26, 35), (13, 32, 14, 33), (11, 30, 12, 31), (36, 5, 37, 6), (4, 35, 5, 36), (9, 28, 10, 29), (19, 0, 20, 1)]

\noindent $\tau(k_1)=-8\Rightarrow u(k_1) \geq 8$

\noindent PD code of second summand $K_2$: 

\noindent [(9, 25, 10, 24), (26, 42, 27, 41), (23, 38, 24, 39), (6, 34, 7, 33), (3, 41, 4, 40), (16, 44, 17, 43), (51, 33, 0, 32), (17, 13, 18, 12), (31, 51, 32, 50), (21, 9, 22, 8), (49, 5, 50, 4), (20, 45, 21, 46), (19, 15, 20, 14), (46, 28, 47, 27), (1, 37, 2, 36), (42, 11, 43, 12), (39, 3, 40, 2), (37, 22, 38, 23), (13, 19, 14, 18), (28, 48, 29, 47), (30, 6, 31, 5), (48, 30, 49, 29), (25, 11, 26, 10), (44, 16, 45, 15), (34, 8, 35, 7), (35, 1, 36, 0)]

\noindent $\tau(K_2)=7\Rightarrow u(K_2) \geq 7$\\
$u(K_1 + K_2) \ge 15$ assuming additivity of unknotting number. \\
Minimal unknotting sequence: 

\noindent [49, 0, 20, 1, 57, 76, 66, 85, 84, 79, 56, 96, 67, 65, 69]  \\
Crossing switches to reach 12a898: 

\noindent [49, 0, 1, 57, 76, 66, 85, 84, 79, 56, 96, 67]

\subsection{12a916}
Initial PD code of $K_1 \# K_2$: 

\noindent [(55, 22, 56, 23), (58, 34, 59, 33), (57, 36, 58, 37), (59, 19, 60, 18), (63, 14, 64, 15), (61, 30, 62, 31), (64, 40, 65, 39), (65, 27, 66, 26), (69, 44, 70, 45), (67, 49, 68, 48), (71, 11, 72, 10), (73, 108, 74, 109), (80, 16, 81, 15), (76, 21, 77, 22), (82, 32, 83, 31), (79, 38, 80, 39), (85, 12, 86, 13), (86, 42, 87, 41), (88, 3, 89, 4), (89, 47, 90, 46), (92, 43, 93, 44), (90, 2, 91, 1), (94, 10, 95, 9), (95, 110, 96, 111), (98, 6, 99, 5), (96, 7, 97, 8), (97, 24, 98, 25), (100, 37, 101, 38), (101, 17, 102, 16), (104, 30, 105, 29), (106, 108, 107, 107), (52, 109, 53, 110), (77, 57, 78, 56), (87, 67, 88, 66), (93, 71, 94, 70), (51, 73, 52, 72), (54, 76, 55, 75), (102, 82, 103, 81), (60, 84, 61, 83), (68, 92, 69, 91), (78, 100, 79, 99), (62, 104, 63, 103), (84, 106, 85, 105), (74, 54, 75, 53), (23, 6, 24, 7), (111, 8, 0, 9), (32, 17, 33, 18), (35, 20, 36, 21), (4, 25, 5, 26), (13, 28, 14, 29), (19, 34, 20, 35), (27, 40, 28, 41), (49, 42, 50, 43), (2, 47, 3, 48), (11, 50, 12, 51), (45, 0, 46, 1)]

\noindent PD code of first summand $K_1$: 

\noindent [(11, 4, 12, 5), (1, 8, 2, 9), (18, 13, 19, 14), (16, 21, 17, 22), (12, 19, 13, 20), (20, 23, 21, 0), (9, 2, 10, 3), (15, 6, 16, 7), (3, 10, 4, 11), (22, 17, 23, 18), (7, 14, 8, 15), (5, 0, 6, 1)]

\noindent $\tau(K_1)=-4\Rightarrow u(K_1) \geq 4$

\noindent PD code of second summand $K_2$:

\noindent [(2, 10, 3, 9), (19, 11, 20, 10), (16, 6, 17, 5), (4, 16, 5, 15), (1, 19, 2, 18), (11, 21, 12, 20), (14, 4, 15, 3), (13, 23, 14, 22), (23, 7, 0, 6), (21, 13, 22, 12), (8, 18, 9, 17), (7, 1, 8, 0)]

\noindent $\tau(K_2)=5\Rightarrow u(K_2) \geq 5$

\noindent $u(K_1 + K_2) \ge 9$ assuming additivity of unknotting number.

\noindent Minimal unknotting sequence: 

\noindent [10, 44, 46, 47, 53, 33, 42, 7, 36]

\noindent Initial crossing switches: 

\noindent [10, 46, 53, 42, 36]

\noindent PD code  after initial switches and simplification: 

\noindent [(5, 22, 6, 23), (9, 27, 10, 26), (19, 8, 20, 9), (14, 23, 15, 24), (11, 3, 12, 2), (13, 5, 14, 4), (16, 8, 17, 7), (1, 11, 2, 10), (3, 13, 4, 12), (6, 16, 7, 15), (24, 17, 25, 18), (27, 20, 0, 21), (18, 25, 19, 26), (21, 0, 22, 1)]

\noindent Crossing switch after simplification: [3]

\noindent Final PD code (12a916): 

\noindent [(6, 23, 7, 24), (8, 26, 9, 25), (18, 9, 19, 10), (22, 16, 23, 15), (11, 3, 12, 2), (13, 5, 14, 4), (16, 8, 17, 7), (1, 11, 2, 10), (3, 13, 4, 12), (5, 15, 6, 14), (24, 17, 25, 18), (27, 20, 0, 21), (19, 26, 20, 27), (21, 0, 22, 1)]

\subsection{12a999}
Initial PD code of $K_1 \# K_2$: 

\noindent [(55, 22, 56, 23), (58, 34, 59, 33), (57, 36, 58, 37), (59, 19, 60, 18), (63, 14, 64, 15), (61, 30, 62, 31), (64, 40, 65, 39), (65, 27, 66, 26), (69, 44, 70, 45), (67, 49, 68, 48), (71, 11, 72, 10), (73, 108, 74, 109), (80, 16, 81, 15), (76, 21, 77, 22), (82, 32, 83, 31), (79, 38, 80, 39), (85, 12, 86, 13), (86, 42, 87, 41), (88, 3, 89, 4), (89, 47, 90, 46), (92, 43, 93, 44), (90, 2, 91, 1), (94, 10, 95, 9), (95, 110, 96, 111), (98, 6, 99, 5), (96, 7, 97, 8), (97, 24, 98, 25), (100, 37, 101, 38), (101, 17, 102, 16), (104, 30, 105, 29), (106, 108, 107, 107), (52, 109, 53, 110), (77, 57, 78, 56), (87, 67, 88, 66), (93, 71, 94, 70), (51, 73, 52, 72), (54, 76, 55, 75), (102, 82, 103, 81), (60, 84, 61, 83), (68, 92, 69, 91), (78, 100, 79, 99), (62, 104, 63, 103), (84, 106, 85, 105), (74, 54, 75, 53), (23, 6, 24, 7), (111, 8, 0, 9), (32, 17, 33, 18), (35, 20, 36, 21), (4, 25, 5, 26), (13, 28, 14, 29), (19, 34, 20, 35), (27, 40, 28, 41), (49, 42, 50, 43), (2, 47, 3, 48), (11, 50, 12, 51), (45, 0, 46, 1)]

\noindent PD code of first summand $K_1$: 

\noindent [(11, 4, 12, 5), (1, 8, 2, 9), (18, 13, 19, 14), (16, 21, 17, 22), (12, 19, 13, 20), (20, 23, 21, 0), (9, 2, 10, 3), (15, 6, 16, 7), (3, 10, 4, 11), (22, 17, 23, 18), (7, 14, 8, 15), (5, 0, 6, 1)]

\noindent $\tau(K_1)=-4 \Rightarrow u(K_1) \geq 4$

\noindent PD code of second summand $K_2$: 

\noindent [(2, 10, 3, 9), (19, 11, 20, 10), (16, 6, 17, 5), (4, 16, 5, 15), (1, 19, 2, 18), (11, 21, 12, 20), (14, 4, 15, 3), (13, 23, 14, 22), (23, 7, 0, 6), (21, 13, 22, 12), (8, 18, 9, 17), (7, 1, 8, 0)]

\noindent $\tau(K_2)=5 \Rightarrow u(K_2) \geq 5$

\noindent  $u(K_1 + K_2) \ge 9$ assuming additivity of unknotting number.

\noindent Minimal unknotting sequence: 

\noindent [10, 44, 46, 47, 53, 33, 42, 7, 36]

\noindent Initial crossing switches: 

\noindent [10, 46, 53, 42, 36]

\noindent PD code after initial switches and simplification: 

\noindent [(5, 26, 6, 27), (6, 30, 7, 29), (21, 10, 22, 11), (15, 22, 16, 23), (14, 2, 15, 1), (12, 4, 13, 3), (16, 8, 17, 7), (18, 10, 19, 9), (2, 12, 3, 11), (4, 14, 5, 13), (8, 18, 9, 17), (27, 20, 28, 21), (31, 24, 0, 25), (19, 28, 20, 29), (23, 30, 24, 31), (25, 0, 26, 1)]

\noindent Crossing switch after simplification: [0]

\noindent Final PD code (12a999): 

\noindent [(5, 26, 6, 27), (6, 30, 7, 29), (21, 10, 22, 11), (15, 22, 16, 23), (14, 2, 15, 1), (12, 4, 13, 3), (16, 8, 17, 7), (18, 10, 19, 9), (2, 12, 3, 11), (4, 14, 5, 13), (8, 18, 9, 17), (27, 20, 28, 21), (31, 24, 0, 25), (19, 28, 20, 29), (23, 30, 24, 31), (25, 0, 26, 1)]

\end{document}